\documentclass[12pt]{amsart}

\usepackage{amsmath,amssymb,amsthm}
\usepackage[colorlinks, bookmarks=true]{hyperref}

\theoremstyle{plain}
\newtheorem{theorem}{Theorem}[subsection]
\newtheorem{corollary}[theorem]{Corollary}
\newtheorem{lemma}[theorem]{Lemma}
\newtheorem{proposition}[theorem]{Proposition}

\theoremstyle{definition}

\newtheorem{remark}[theorem]{Remark}

\newcommand{\norm}[1]{\lVert#1\rVert}

\renewcommand{\leq}{\leqslant}
\renewcommand{\geq}{\geqslant}\usepackage{amssymb}

\newcommand{\vr}{\varepsilon}
\newcommand{\one}{\mathbf{1}}

\newcommand{\be}{\begin{equation}}
\newcommand{\ee}{\end{equation}}

\newcommand{\R}{\mathbb R}
\newcommand{\Z}{\mathbb Z}
\newcommand{\N}{\mathbb N}
\newcommand{\C}{\mathbb C}

\newcommand{\schu}{\mathbf S}
\newcommand{\pp}{\mathbf P}
\newcommand{\is}{\mathcal S}

\newcommand{\trace}{\mathrm{tr}}
\newcommand{\ce}{\mathcal E}
\newcommand{\se}{{\mathcal S}_{\mathcal{E}}}

\newcommand{\ball}{\mathbf{B}}
\newcommand{\A}{\mathcal{A}}

\newcommand{\pprec}{\prec \!\! \prec}
\newcommand{\gen}{\mathbf{G}}
\newcommand{\nor}{\mathbf{N}}
\newcommand{\psol}{\mathbf{PSol}}

\def \rank{{\mathrm{rank}} \, }
\def \ker{{\mathrm{ker}} \, }
\def \tr{{\mathrm{Tr}} }

\renewcommand{\span}{\mathrm{span}}
\def \vr{\varepsilon}

\def \ran{{\mathrm{ran}} \, }

\def \eqalign#1{\null\,\vcenter{\openup\jot 
   \ialign{\strut\hfil$\displaystyle{##}$&$
      \displaystyle{{}##}$\hfil \crcr#1\crcr}}\,}

\DeclareMathOperator{\re}{Re}

\begin{document}

\baselineskip=16pt

\numberwithin{equation}{section}

\pagestyle{headings}

\title[Domination of operators]{Domination of operators in the non-commutative
setting}

\author[T. Oikhberg]{Timur Oikhberg}
\address{
Dept.~of Mathematics, University of Illinois at Urbana-Champaign, Urbana IL 61801, USA}
\email{oikhberg@illinois.edu}

\author[E. Spinu]{Eugeniu Spinu}
\address{Dept. of Mathematical and Statistical Sciences, University of Alberta
Edmonton, Alberta  T6G 2G1, CANADA}
\email{espinu@ualberta.ca}

\begin{abstract}
We consider majorization problems in the non-commutative setting.
More specifically, suppose $E$ and $F$ are ordered normed spaces
(not necessarily lattices), and $0 \leq T \leq S$ in $B(E,F)$.
If $S$ belongs to a certain ideal (for instance, the ideal of compact
or Dunford-Pettis operators), does it follow that $T$ belongs to that
ideal as well? We concentrate on the case when $E$ and $F$ are
$C^*$-algebras, preduals of von Neumann algebras, or non-commutative
function spaces.
In particular, we show that, for $C^*$-algebras $\A$ and ${\mathcal{B}}$,
the following are equivalent: (1) at least one of the two conditions holds:
(i) $\A$ is scattered, (ii) ${\mathcal{B}}$ is compact;
(2) if $0 \leq T \leq S : \A \to {\mathcal{B}}$, and $S$ is compact,
then $T$ is compact.
\end{abstract}

\thanks{The authors acknowledge the generous support of Simons Foundation,
via its travel grant. The first author also benefitted from a COR grant of
the UC system. The authors would also like to thank the referee for a
careful reading of the paper, and making valuable suggestions.}


\maketitle



\section{Preliminaries}\label{sec:first}

\subsection{Introduction}\label{ss:intro}

Following \cite[Definition II.1.2]{Sch},
we say that a real Banach space $Z$ is an \emph{ordered Banach space}
(\emph{OBS} for short) if it is equipped with
a positive cone $Z_+$, closed in the norm topology. Throughout, we
assume that $Z_+$ is \emph{proper} (or \emph{pointed}) -- that is,
$Z_+ \cap (-Z_+) = \{0\}$.
The positive cone of an OBS $Z$ is called
\emph{generating} if $Z_+ - Z_+ = Z$. Equivalently (see \cite{An}, \cite{BR}),
there exists $\gen_Z$ (the \emph{generating constant} of $Z$) so that,
for any $z \in Z$, there exist $a, b \in Z_+$ so that $z = a - b$,
and $\max\{\|a\|, \|b\|\} \leq \gen_Z \|z\|$. Abusing the notation slightly,
we call such OBSs generating. 
We say that an OBS $Z$ is \emph{normal} 
if there exists $\nor_Z$ (the \emph{normality constant} of $Z$) so that
$\|z\| \leq \nor_Z (\|a\| + \|b\|)$ whenever $a \leq z \leq b$.
By \cite[Section 1.1]{BR} or \cite{An}, $Z$ is normal iff its dual $Z^\star$
is generating, and vice versa.

In the current article we consider the following question.
Suppose $0 \le T \le S$ are operators acting between two ordered
Banach spaces, and $S$ belongs to a certain class of operators (say, compact
or Dunford-Pettis). Does this imply that $T$ belongs to the same class?
This question is usually referred as the \emph{Domination Problem}.
For arbitrary ordered normed spaces, the set-up may be too general
to obtain meaningful results. In the (rather restrictive) setting
of operators between Banach lattices, the Domination Problem 
has been widely investigated (see e.g. \cite{AB:80},
\cite{AB:81}, \cite{DDFr}, \cite{Sp}, \cite{FHT}, \cite{KaSa}, \cite{Wi}).

We concentrate on the non-commutative version of the Domination Problem.
More specifically, we consider the case when the domain and/or
range of the operators involved is either a $C^*$-algebra, its dual or
predual, or a non-commutative function space. We refer the reader to
e.g.~\cite{DDdP93}, or to the survey article \cite{deP}, for the
definition of the latter. Here, we only briefly outline the basic
properties of such spaces.

Suppose a von Neumann algebra $\A$ is equipped with a normal faithful
semi-finite trace $\tau$. An operator $x$ is called \emph{$\tau$-measurable}
if it is (i) closed and densely defined; (ii) affiliated with $\A$, in the sense that
$u x = x u$ for any unitary $u \in \A^\prime$; and (iii) for some $c > 0$,
the spectral projection $\chi_{(c,\infty)}(|x|)$ has finite trace.
On the set $\tilde{\A}$ of $\tau$-measurable operators,
we define the \emph{generalized singular value function}:
for $x \in \A$ and $t \geq 0$,
$\mu_x(t) = \inf\{\|xe\| : e \in \pp(\A), \tau(e^\perp) \leq t\}$
(see e.g. \cite{deP}, \cite{FK} for other formulae for $\mu_x( \cdot )$).
Here and below, $\pp(\A)$ stands for the set of all projections in $\A$.

Now suppose $\ce$ is a linear subset of $\tilde{\A}$, complete in its
own norm $\| \cdot \|_{\ce}$. We say that $\ce$ is an \emph{non-commutative
function space} if:
\begin{enumerate}
\item
$L_1(\tau) \cap \A \subset \ce \subset L_1(\tau) + \A$.
\item
For any $x \in \ce$ and $a, b \in \A$, we have $axb \in \ce$,
and $\|a x b\|_\ce \leq \|a\| \|x\|_\ce \|b\|$.
\end{enumerate}
$\ce$ is called \emph{symmetric} if, whenever $x \in \ce$,
$y \in \tilde{\A}$, and $\mu_y \leq \mu_x$, then $y \in \ce$, with
$\|y\|_\ce \leq \|x\|_\ce$.
Following \cite{DdP12}, we say that $\ce$ is \emph{strongly symmetric}
if, in addition, for any $x, y \in \ce$, with $y \pprec x$, we have
$\|y\|_\ce \leq \|x\|_\ce$. Here, $\pprec$ refers
to the \emph{Hardy-Littlewood domination}: for any $\alpha > 0$,
$\int_0^\alpha \mu_y(t) \, dt \leq \int_0^\alpha \mu_x(t) \, dt$.
It is known that, as in the commutative case, $y \pprec x$ iff there exists an
operator $T$, contractive both on $\A$ and $\A_\star = L_1(\tau)$,
so that $y = Tx$ \cite{DDdP92}. We say that $\ce$
is \emph{fully symmetric} if it is strongly symmetric and, for any
$x \in \ce$ and $y \in \tilde{A}$, we have $y \in \ce$ whenever
$y \pprec x$.

A non-commutative function space is said to be {\emph{order continuous}}
if, for any sequence $x_n \downarrow 0$, we have $\lim_n \|x_n\| = 0$.
Emulating the proof of \cite[Proposition 1.a.8]{LT2}, one proves that this
is equivalent to requiring that, for any net $x_\alpha \downarrow 0$,
$\lim_\alpha \|x_\alpha\| = 0$.

Note that, if  $-a \le b \le a$ for $a,b \in \tilde{\A}$, then $\mu_b \leq \mu_a$.
Indeed,  pick $t \in \R$ and $\lambda > \mu_a(t)$. Set $e = \chi_{[0,\lambda]}(a)$.
Then $\tau(e^\perp) \leq t$. Furthermore, $eae \geq ebe \geq -eae$, hence
$\mu_b(t) \leq \|ebe \| \leq \|eae\| \leq \lambda$. Taking the infimum over $\lambda$,
we obtain $\mu_b \leq \mu_a$.

Consequently, if $a,b \in \ce$ satisfy $-a \leq b \leq a$, then $\|b\| \le \|a\|$.
Therefore, $\ce$ is normal with constant $2$. It is also easy to see that $\ce$ is
generating with constant $2$. Consequently, the duals of $\ce$ of all orders
are both generating and normal.

Many symmetric non-commutative function spaces arise from their commutative
analogues. Indeed, suppose $\tau$ is a normal faithful semi-finite trace
on a von Neumann algebra $\A$. It is known that  if $\A$ has no atomic projections,
then the range of $\tau$ (denoted by $\Omega = \Omega_\tau$) is $[0,\tau(\one)]$
(with $\tau(\one) < \infty$), or $[0,\infty)$. On the other hand, if $\A$ is atomic
(that is, any projection has a minimal subprojection), then $ \Omega_\tau$ is either
$\{0,1,\ldots,n\}$ or $\Z_+ = \{0,1,2,\ldots\}$.  Suppose
$\ce$ is a symmetric function space (in the sense of e.g. \cite{KPS}) on $\Omega$.
We can define the corresponding non-commutative function space $\ce(\tau)$,
consisting of those $x \in \tilde{A}$ for which the norm
$\|x\|_{\ce(\tau)} = \|\mu_x\|_\ce$ is finite. By \cite{KS}, this procedure
yields a Banach space. It is well known
(see e.g. \cite{DDdP93}, \cite{DdP12}, \cite{deP}) that many properties of the
function space $\ce$ (for instance, being reflexive or order continuous)
pass to the non-commutative space $\ce(\tau)$. 

In the discrete case ($\ce$ is a symmetric sequence space on $\N$, and
$\tau$ is the canonical trace on $B(H)$), the construction above
produces a \emph{non-commutative symmetric sequence space}
(often referred to as a \emph{Schatten space}),
denoted by $\se(H)$ (instead of $\ce(\tau)$). 
When $H = \ell_2$ ($H = \ell_2^n)$, we write $\se$ (resp.~$\se^n$)
instead of $\se(H)$. For properties of Schatten spaces, the reader is
referred to e.g.~\cite{GK}, \cite{Si}.
We must note that any separable symmetric non-commutative sequence
space arises from a sequence space \cite[Section III.6]{GK}.

Note that a symmetric function (or sequence) space is separable iff it is
order continuous. Indeed, symmetric function spaces are order complete,
and, for such spaces, separability implies order continuity
\cite[Proposition 1.a.7]{LT2}.
On the other hand, it is well
known that any non-negative function is a limit (a.e.) of an increasing
sequence of simple functions. Thus, by \cite[Theorem II.4.8]{KPS}, any
order continuous symmetric function space is separable. Furthermore,
by \cite[Theorem II.4.10 and its Corollary]{KPS}, such spaces are fully symmetric
(equivalently, they are interpolation spaces between $L_1$ and $L_\infty$).
Some non-commutative generalizations of these results are contained
in \cite{DdP11}.

Surprisingly, the non-commutative Domination Problem has attracted
little attention so far.
The connections between domination and irreducibility (for maps between
von Neumann algebras) were studied in \cite{Fa}. In \cite{Mu}, domination of linear
functionals on Banach $*$-algebras was used to obtain automatic continuity results.
Domination of completely positive compact operators has recently
been investigated in \cite{DdP10}.

The paper is structured as follows. First (Section \ref{sec:first}), we prove some
preliminary results about the properties of positive operators, order intervals,
and positive solids. In Subsection \ref{ss:compact}, we establish some
basic facts about non-commutative function spaces. In Subsection \ref{ss:intervals},
we investigate compact $C^*$-algebras, characterizing them in terms of compactness
of order intervals. We also show that a $C^*$-algebra is compact iff it is
hereditary in its enveloping algebra. Subsection \ref{ss:psp} deals with the
positive analogues of the Schur Property. In Subsection \ref{ss:psp predual},
we study compactness of order intervals in preduals of von Neumann algebras.

Our main results are contained in Section \ref{sec:main}. In Subsection
\ref{ss:schatten_main}, we investigate whether an operator to or from a
non-commutative function space, dominated by a compact operator, must
itself be compact. Subsection \ref{ss:C*alg} is devoted to the same question
for $C^*$-algebras. In Subsection \ref{ss:mult_ops}, we consider domination
by compact multiplication operators on $C^*$-algebras. In Subsection \ref{ss:DP_mult},
we tackle domination properties of Dunford-Pettis Schur multipliers.
Subsection \ref{ss:WC} is devoted to the domination properties of
weakly compact operators.

Other classes of operators are considered in Section \ref{sec:misc}.
In Subsection \ref{ss:positive},
we show that complete positivity and decomposability are not preserved under
domination. Subsection \ref{ss:op_sys} demonstrates that operator systems
have too little structure to meaningfully consider domination.

Throughout the paper, we use standard Banach space results and notation.
If $a$ is a (closed densely defined) operator, $a^*$ refers to the
adjoint of $a$. The same notation is used in preduals of von Neumann algebras.
If $E$ is a Banach space, $E^\star$ refers to its dual. Similar notation
is used for the predual, and for the conjugate of an operator
between Banach spaces. $\ball(E)$
stands for the unit ball of $E$. If $S$ is a subset of an ordered Banach
space, we denote by $S_+$ the intersection of $S$ with the positive cone.
We denote by $\ce^\times$ the \emph{ K\"othe dual} of a non-commutative
symmetric function space $\ce$  (see e.g. \cite{DDdP93}, \cite{deP} for the
definition and the basic properties of K\"othe duals).


\subsection{Compactness and positivity in Schatten spaces}\label{ss:compact}

To work with Schatten spaces, we need to introduce some notation.
Denote the canonical basis in $\ell_2$ by $(e_k)$.
Let $P_n$ be the orthogonal projection onto
$\span[e_1, \ldots, e_n]$, and $P_n^\perp = \one - P_n$. For convenience, set $P_0 = 0$.
If $\ce$ is a non-commutative symmetric sequence space, let $Q_n$ be
the projection on $\ce$, defined via $Q_n x = P_n x P_n$. Similarly,
let $R_n x = P_n^\perp x P_n^\perp$.

\begin{lemma}\label{lem:maps from S_E}
Suppose $\ce$ is a non-commutative symmetric sequence space on $B(\ell_2)$,
$Z$ is an ordered normed space, and
$T : \ce \to Z$ is a positive operator. Then, for any $x \in \ce_+$,
$\|T(x - R_n x - Q_n x)\|^2 \leq 16 \nor_Z \|T(Q_n x)\| \|T(R_n x)\|$, where
$\nor_Z$ is the normality constant of $Z$.
If $Z$ is a non-commutative symmetric function space, then
$\|T(x - R_n x - Q_n x)\|^2 \leq 4 \|T(Q_n x)\| \|T(R_n x)\|$.
\end{lemma}

\begin{proof}
For $t \in \R \backslash \{0\}$, consider $U(t) = t P_n + t^{-1} P_n^\perp$, and
$V(t) = t P_n - t^{-1} P_n^\perp$. These operators are self-adjoint and invertible,
hence $x(t) = U(t) x U(t)$ and $y(t) = V(t) x V(t)$ are positive elements of $\ce$.
An elementary calculation shows that $x(t) = t^2 Q_n x + t^{-2} R_n x + (x - Q_n x - R_n x)$,
and $y(t) = t^2 Q_n x + t^{-2} R_n x - (x - Q_n x - R_n x)$.
Let $a(t) = t^2 Q_n x + t^{-2} R_n x$,
and $b = x - Q_n x - R_n x$. By the above, $-a(t) \leq b \leq a(t)$. Therefore, for any $t$,
$$
\frac{1}{2\nor_Z} \|Tb\| \leq \|T a(t)\| \leq t^2 \|T Q_n x\| + t^{-2} \|T R_n x\| .
$$
Taking $t = \|T R_n x\|^{1/4}/\|T Q_n x\|^{1/4}$, we obtain the desired inequality.
If, in addition, $Z$ is a non-commutative symmetric function space, then $\|Tb\| \leq \|T a(t)\|$.
\end{proof}

\begin{corollary}\label{cor:maps from S_E}
Suppose $\ce$ is a non-commutative symmetric sequence space on $B(\ell_2)$,
$Z$ is a normal OBS, and $T : \ce \to Z$ is a positive operator. Then
$$
\|T (I - Q_n)\| \leq
 \|T R_n\| + 16 { \nor_Z}^{1/2} \|T R_n\|^{1/2} \|T Q_n\|^{1/2} .
$$
If $Z$ is a non-commutative symmetric function space, then
$\|T (I - Q_n)\| \leq \|T R_n\| + 8 { \nor_Z}^{1/2} \|T R_n\|^{1/2} \|T Q_n\|^{1/2}$.
\end{corollary}

\begin{proof}
We prove the corollary for general $Z$ (the case of $Z$ being a non-commutative function
space follows with minor modifications).
Lemma \ref{lem:maps from S_E} shows that, for $x \geq 0$,
$$\|T(I - R_n - Q_n) x\| \leq 4 \nor_Z^{1/2} \|T R_n\|^{1/2} \|T Q_n\|^{1/2} \|x\|.$$
A polarization argument implies
$\|T(I - R_n - Q_n)\| \leq 16 \nor_Z^{1/2} \|T R_n\|^{1/2} \|T Q_n\|^{1/2}$.
By the triangle inequality,
$\|T (I - Q_n)\| \leq \|T R_n\| + \|T(I - R_n - Q_n)\|$.
\end{proof}

For future use, we need to quote a result from \cite[Section 2]{CS}.

\begin{lemma}\label{lem:ord cont}
Suppose $\tau$ is a normal faithful semi-finite trace on a von Neumann
algebra $\A$, and a strongly symmetric non-commutative function space $\ce$
is order continuous.
Suppose, furthermore, that $x$ is an element of $\A$, and a sequence
of projections $p_n \in \A$ decreases to $0$ in the strong operator topology.
Then $\lim_n \|x p_n\| = \lim_n \|p_n x\| = \lim_n \|p_n x p_n\| = 0$.
\end{lemma}




Specializing to Schatten spaces, we obtain:


\begin{corollary}\label{cor:ord cont}
Suppose $\ce$ is an order continuous 
symmetric sequence space. 
Then, for every $x \in \se$, $\lim_n \|x - Q_n x\| = 0$.
\end{corollary}

\begin{proof}
By \cite[Section 3]{DDdP93}, $\se$ is order continuous iff $\ce$
is order continuous. It suffices to show that, for $x \in \ball(\se)_+$,
and $\vr \in (0,1)$, $\|x - Q_n x\| < \vr$
for $n$ sufficiently large. 
This follows from  the estimate
$\|x-Q_n x\|=\|P_n^\perp xP_n+xP_n^\perp\| \le \|P_n^\perp x\| +\|xP_n^\perp\|$
and Lemma \ref{lem:ord cont}.
\end{proof}

\begin{lemma}\label{lem:diagonal}
Suppose $\ce$ is an order continuous symmetric 
sequence space, not containing $\ell_1$, and $S : \se \to Z$ is compact
($Z$ is a Banach space). Then $\lim_n \|S|_{R_n(\se)}\| = 0$.
\end{lemma}

\begin{proof}
Suppose not. By Corollary \ref{cor:ord cont}, we have $\lim_n \|(I - Q_n) x\| = 0$.
A standard approximation argument yields a sequence
$0 = n_0 < n_2 < \ldots$ with the property that for each $k$
there exists $x_k \in \se$, so that $\|x_k\| = 1$, and
$(P_{n_k} - P_{n_{k-1}}) x_k (P_{n_k} - P_{n_{k-1}}) = x_k$, and
$\|S x_k\| > c > 0$. By compactness, the sequence $(S x_k)$ must have a convergent
subsequence $(S x_{k_i})$. Then $\lim_N N^{-1} \|\sum_{i=1}^N S x_{k_i}\| > 0$, while
$\lim_N N^{-1} \|\sum_{i=1}^N x_{k_i}\| = 0$.
\end{proof}

Next we describe the Schatten spaces not containing $\ell_1$.

\begin{proposition}\label{prop:l_1_S_E_cont}
Let $\ce$ be a separable symmetric sequence space.
For any infinite-dimensional Hilbert space $H$,
the following are equivalent:
\begin{enumerate}
\item $\ce$ contains a copy of $\ell_1$.
\item $\ce$ contains a lattice copy of $\ell_1$ positively complemented.
\item $\se(H)$ contains a positively complemented copy of $\ell_1$ spanned
by a disjoint positive sequence.
\item $\se(H)$ contains a copy of $\ell_1$.
\end{enumerate}
\end{proposition}

\begin{proof}
The implications (2) $\Rightarrow$ (1) and
(3)  $\Rightarrow$ (4) are trivial.
To show (2) $\Rightarrow$ (3), observe that $\se(H)$ contains $\ce$
as a diagonal subspace, which is positively complemented.
(4) $\Rightarrow$ (1) follows directly from \cite[Corollary 3.2]{Ar}.
To prove (1) $\Rightarrow$ (2), apply a
``gliding hump'' argument to show that $\ce$ contains disjoint vectors
$(x_i)$, equivalent to the canonical basis of $\ell_1$.
Then $X = \span[|x_i| : i \in \N]$ is a sublattice of $\ce$,
lattice isomorphic to $\ell_1$. By \cite[Theorem 2.3.11]{M-N},
$X$ is positively complemented.
\end{proof}

For a subset $M \subset  X_{+}$ ($X$ is an OBS), define 
the \emph{positive solid} of $M$:
 $$\psol(M)=\{ x \in X_{+},\text{ such that } 0\le x \le y \text{ and } y \in M\}.$$

\begin{lemma}\label{lem:psol comp}
If $\ce$ is an order continuous non-commutative symmetric sequence space,
and $M \subset \ce$ is relatively compact, then $\psol(M)$ is relatively compact.
In particular, any order interval in an order continuous non-commu\-ta\-tive symmetric
sequence space is compact.
\end{lemma}

For the proof, we need two technical results.

\begin{lemma}\label{lem:sep}
Suppose $\ce$ and $M$ are as in Lemma \ref{lem:psol comp}. Then there exists a
projection $p$ with separable range, so that $M = pMp$.
\end{lemma}

\begin{proof}
The set $M$ must contain a countable dense subset $S$. The elements
of $M$ are compact operators, hence, for any $x \in S$, there exists
a projection $p_x$ with separable range, so that $p_x x p_x = x$.
Then $p = \vee_{x \in S} p_x$ has the desired properties.
\end{proof}

\begin{lemma} \label{aprox}
Suppose $\ce$ is an order continuous non-commutative symmetric
sequence space on $B(\ell_2)$, and $M$ is relatively compact
subset of $\ce$. 
Then $\lim_n \|R_n|_{M}\| = 0$.
\end{lemma}

\begin{proof}
For every $\vr>0$  there are $x_1, \ldots, x_k$ in $M$ such that for every $x \in M$
there is an $1\le i \le k$ such that $\|x-x_i\| <\vr/2$.
Pick $N \in \mathbb{N}$ such that $\|R_nx_i\| <\epsilon/2$ for every $n>N$ and
$ 1\le i \le k$.
Hence, $\|R_nx\|  \le \|R_nx_i\| +\|R_n\|\|x-x_i\|<\epsilon$ for every $x \in M$ and $n>N$.
\end{proof}

\begin{proof}[Proof of Lemma \ref{lem:psol comp}]
By Lemma \ref{lem:sep}, we can restrict ourselves to spaces on $B(\ell_2)$.
As $Q_n$ is a finite rank projection, it suffices to show that,
for any $\vr \in (0,1)$,
there exists $n \in \N$ so that $\|(I - Q_n) x\| < \vr$ for any $x \in \psol(M)$.
To this end, write $(I - Q_n)x = (x - Q_n x - R_n x) + R_n x$.
Reasoning as in the proof of Lemma \ref{lem:maps from S_E}, we observe that
$$
- \big( t^2 Q_n x + t^{-2} R_n x \big)
 \leq x - Q_n x - R_n x \leq t^2 Q_n x + t^{-2} R_n x
$$
for any $t > 0$, hence
$\|x - Q_n x - R_n x\| \leq t^2 \|Q_n x\| + t^{-2} \|R_n x\|$.
Taking $t = \|R_n x\|^{1/2}/\|Q_n x\|^{1/2}$, we obtain
$\|x - Q_n x - R_n x\| \leq 2 \|R_n x\|^{1/2} \|Q_n x\|^{1/2}$.

By scaling, we can assume that $\sup_{y \in M} \|y\| = 1$. By Lemma \ref{aprox},
there exists $n \in \N$ so that $\|R_n y\| < \vr^2/16$ for any $y \in M$.
For any $x \in \psol(M)$, there exists $y \in M$ so that $0 \leq x \leq y$, hence
$0 \leq R_n x \leq R_n y$. By the above,
$\|x - Q_n x - R_n x\| \leq 2 \|R_n y\|^{1/2} < \vr/2$, hence
$$
\|(I - Q_n)x\| = \|x - Q_n x - R_n x\| + \|R_n x\| \leq
 \frac{\vr}{2} + \frac{\vr^2}{16} < \vr .
$$
\end{proof}

Recall that if $Z$ is an OBS, and $x \in Z_+$, the \emph{order interval}
$[0,x]$ is the set $\{y \in Z_+ : y \leq x\}$.

\begin{corollary}\label{cor:crit OC}
Suppose $\ce$ is a fully symmetric non-commutative sequence space.
Then $\ce$ is order continuous if and only if any order interval
in $\ce$ is compact.
\end{corollary}

\begin{lemma}\label{lem:+l_infty}
Suppose $\ce$ is a fully symmetric non-commutative function or
sequence space, which is not order continuous. Then there exists
a positive complete isomorphism $j : \ell_\infty \to \ce$.
\end{lemma}

\begin{proof}
In the notation of \cite[Section 6]{DdP12}, there exists
$x \in \ce_+ \backslash \ce^{an}$. Moreover, there exists
a sequence of mutually orthogonal projections $e_i \in \A$
($i \in \N$), so that $\inf_i \|e_i x e_i\| > 0$. The map
$y \mapsto \sum_i e_i y e_i$ is contractive in $\A$, and
in its predual, hence $\sum_i e_i y e_i \pprec y$, for any
$y \in \A + \A_\star$. Due to $\A$ being fully symmetric,
$\sum_i e_i x e_i \in \ce$, and $\|\sum_i e_i x e_i\| \leq \|x\|$.
Therefore, the map
$$
j : \ell_\infty \to \ce : (\alpha_i) \mapsto
 (\sum_i \alpha_i e_i) (\sum_i e_i x e_i) = \sum_i \alpha_i e_i x e_i
$$
has the desired properties.
\end{proof}

\begin{proof}[Proof of Corollary \ref{cor:crit OC}]
Note that an order interval $[0,x]$ is closed. If $\ce$
is order continuous, an application of Lemma \ref{lem:psol comp}
to $M = \{x\}$ shows the compactness of $[0,x]$. 
If $\ce$ is not order continuous, then, for $x$ as in
Lemma \ref{lem:+l_infty}, $[0,x]$ is not (relatively) compact.
\end{proof}



\subsection{Compactness of order intervals in $C^*$-algebras}
\label{ss:intervals}

In this subsection, we investigate the compactness of order intervals
in $C^*$-algebras, and obtain a new description of compact $C^*$-algebras.

First we introduce some definitions.
We say that an element $a$ of a Banach algebra $\A$
is \emph{multiplication compact}
if the map $\A \to \A : b \mapsto aba$ is compact. Combining
\cite{Yl75PAMS}, \cite{Yl75Suomi}, we see that, for
an element $a$ of a $C^*$-algebra $\A$, the following are equivalent:
\begin{enumerate}
\item
$a$ is multiplication compact.
\item
The map $\A \to \A : b \mapsto ab$ is weakly compact.
\item
The map $\A \to \A : b \mapsto ba$ is weakly compact.
\item
The map $\A \to \A : b \mapsto aba$ is weakly compact.
\end{enumerate}
By \cite{Yl72PAMS}, there exists a faithful representation
$\pi : \A \to B(H)$ so that $a$ is multiplication compact iff
$\pi(a)$ is a compact operator on $H$. If, in addition, $\A$ is
an irreducible $C^*$-subalgebra of $B(H)$, then $a \in \A$
is multiplication compact iff $a$ is a compact operator
\cite{Yl68Suomi}.

Suppose $\A$ is a $C^*$-subalgebra of $B(H)$, where $H$ is a Hilbert space.
For $x \in B(H)$ we define an operator $M_x : \A \to B(H) : a \mapsto x^* a x$.

\begin{lemma}\label{lem:M_a}
For an element $a$ of a $C^*$-algebra $\A$, the following are equivalent.
\begin{enumerate}
\item
$a$ is multiplication compact.
\item
The operator $M_a$ is compact.
\item
The operator $M_a$ is weakly compact.
\end{enumerate}
\end{lemma}

\begin{proof}
$(2) \Rightarrow (3)$ is trivial. To show $(1) \Rightarrow (2)$,
recall that $a$ is multiplication compact iff the map
$\A \to \A : b \mapsto ab$ is weakly compact. Passing to the
adjoint, we see that the last statement holds iff the map
$\A \to \A : b \mapsto ba^*$ is weakly compact, or equivalently,
iff $a^*$ is multiplication compact. By \cite{BT}, this implies
the compactness of $M_a$.

To prove $(3) \Rightarrow (1)$, note that 
$M_a^{\star\star}$ takes $b \in \A^{\star\star}$ to $a^* b a$.
We identify
$M_a^{\star\star}$ with $M_a$, acting on $\A^{\star\star}$. Write
$a = c u$, where $c = (a a^*)^{1/2}$, and $u$ (respectively, $u^*$)
is a partial isometry from $(\ker a)^\perp = (\ker c)^\perp$ to
$\overline{\ran a} = \overline{\ran c}$ (from 
$\overline{\ran a^*} = \overline{\ran c}$ to
$(\ker a^*)^\perp = (\ker c)^\perp$).
Then $M_a = M_u M_c$, and $M_u$ is an isometry on
$\ran(M_c) \subset \A^{\star\star}$. Writing $M_c = M_u^{-1} M_a$,
we conclude that $M_c$ is weakly compact. However,
$M_c x = cxc$, hence, by the remarks preceding the lemma,
$c$ is multiplication compact.
The operator $S : \A^{\star\star} \to \A^{\star\star} : b \mapsto a b a$ can
be written as $S = U M_c V$, where $V b = u b$ and $U b = b u$.
Then $S$ is weakly compact, and therefore, $a$ is multiplication compact.
\end{proof}

Multiplication compactness of elements of a $C^*$-algebra can be described in
terms of compactness of order intervals.

\begin{proposition}\label{prop:comp_int}
For a positive element $a$ of a $C^*$-algebra $\A$, the following are
equivalent:
\begin{enumerate}
\item
$a$ is multiplication compact.
\item
$a^\alpha$ is multiplication compact for any $\alpha > 0$.
\item
The order interval $[0,a]$ 
is compact.
\item
The order interval $[0,a]$ 
is weakly compact.
\end{enumerate}
\end{proposition}

\begin{proof}
The implications (2) $\Rightarrow$ (1) and (3) $\Rightarrow$ (4)
are immediate. To establish (1) $\Rightarrow$ (2), pick a
faithful representation $\pi$ so that $a$ is multiplication compact
if and only if $\pi(a)$ is compact, and note that the compactness 
of $\pi(a)$ is equivalent to the compactness of
$\pi(a)^\alpha = \pi(a^\alpha)$. 

For (2) $\Rightarrow$ (3), assume $\|a\| = 1$.
By \cite[Lemma I.5.2]{Da},
for any $x \in [0,a]$ there exists $u \in \ball(\A)$,
so that $x^{1/2} = u a^{1/4}$, hence $x = a^{1/4} u^* u a^{1/4}$.
In particular, $[0,a] \subset M_{a^{1/4}}(\ball(\A))$.
If $a$ is multiplication compact, then so is $a^{1/4}$.
Therefore, $[0,a]$ is compact.

To prove (4) $\Rightarrow$ (1), suppose $a$ is not
multiplication compact. Then $a^{1/2}$ is not multiplication
compact, hence $M_{a^{1/2}}(\ball(\A))$ is not relatively compact.
Note that any element $x \in \ball(\A)$ can be written as
$x = x_1 - x_2 + i (x_3 - x_4)$, with $x_1, x_2, x_3, x_4 \in \ball(\A)_+$.
Thus, $M_{a^{1/2}}(\ball(\A)_+)$ is not relatively weakly compact. However,
$[0,a] \supset M_{a^{1/2}}(\ball(\A)_+)$. Indeed, if $0 \leq y \leq 1$,
then $0 \leq a^{1/2} y a^{1/2} \leq a$. Therefore, $[0,a]$
is not relatively weakly compact.
\end{proof}

These results allow us to obtain new characterizations of compact
$C^*$-algebras. Recall that a Banach algebra is called \emph{compact} (or
\emph{dual}) if all of its elements are multiplication compact.
By \cite{Al}, compact $C^*$-algebras are precisely the
algebras of the form $\A = (\sum_{i \in I} K(H_i))_{c_0}$, where
each $H_i$ is a complex Hilbert space, and $K(H)$ denotes the
space of compact operators on $H$.
Several alternative characterizations of
compact $C^*$-algebras can be found in \cite[4.7.20]{Dix}.

\begin{proposition}\label{prop:dual_alg}
For a $C^*$-algebra $\A$, the following four statements are equivalent.
\begin{enumerate}
\item
$\A$ is compact.
\item
For any $c \in \A_+$, the order interval $[0,c]$ is compact.
\item
For any $c \in \A_+$, the order interval $[0,c]$ is weakly compact.
\item
For any relatively compact $M \subset \A_+$, $\psol(M)$ is relatively compact.
\end{enumerate}
\end{proposition}

\begin{proof}
The implications (4) $\Rightarrow$ (2) $\Rightarrow$ (3) are immediate.

(3) $\Rightarrow$ (1): by Proposition \ref{prop:comp_int}, any
positive $a \in \A$ is multiplication compact. By
\cite[Corollary 10.4]{BT}, the map $\A \to \A : x \mapsto axb$ 
is compact for any $a, b \in \A_+$. As any $x \in \A$ is a linear
combination of four positive elements, it is multiplication compact.

(1) $\Rightarrow$ (4): it suffices to show that, for any $\vr > 0$,
$\psol(M)$ admits a finite $\vr$-net. Assume, without loss of generality,
that $M \subset \ball(\A)_+$. The map $\A_+ \to \A_+ : a \mapsto a^{1/4}$
is continuous, hence $M^{1/4} = \{a^{1/4} : a \in M\}$ is compact.
Pick $(a_i)_{i=1}^n \subset M$ so that $(a_i^{1/4})_{i=1}^n$ is an
$\vr/4$-net in $M^{1/4}$. By Proposition \ref{prop:comp_int},
$a_i^{1/4}$ is multiplication compact for each $i$, hence
$a_i^{1/4} \ball(\A)_+ a_i^{1/4}$ contains an $\vr/4$-net $(b_{ij})_{j=1}^m$.

Now consider $x \in [0,a]$, for some $a \in M$.
As noted in the proof of Proposition \ref{prop:comp_int}, there exists
$u \in \ball(\A)$, so that $x = a^{1/4} u^* u a^{1/4}$.
Pick $i$ and $j$ so that $\|a^{1/4} - a_i^{1/4}\| < \vr/4$, and
$\|a_i^{1/4} u^* u a_i^{1/4} - b_{ij}\| < \vr/4$. Then
$$   \eqalign{
&
\|a^{1/4} u^* u a^{1/4} - b_{ij}\| \leq
\|(a_i^{1/4} - a^{1/4}) u^* u a^{1/4}\|
\cr
&
+
\|a_i^{1/4} u^* u (a_i^{1/4} - a^{1/4})\| +
\|a_i^{1/4} u^* u a_i^{1/4} - b_{ij}\| < \vr .
}  $$
\end{proof}

Recall that 
a $C^*$-subalgebra $\A$ of a $C^*$-algebra ${\mathcal{B}}$ is called
\emph{hereditary} if, for any $a \in \A_+$, we have
$\{b \in {\mathcal{B}} : 0 \leq b \leq a\} \subset \A$.

\begin{proposition}\label{prop:alg_ideal}
A $C^*$-algebra $\A$ is a hereditary subalgebra of $\A^{\star\star}$
if and only if $\A$ is a compact $C^*$-algebra.
\end{proposition}

\begin{proof}
If $\A$ is compact, then it is an ideal in $\A^{\star\star}$ \cite{Yl75PAMS}.
It is well known (see e.g. \cite[Proposition II.5.3.2]{Bl})
that any ideal in a $C^*$-algebra is hereditary.

Now suppose $\A$ is a hereditary subalgebra of $\A^{\star\star}$.
By \cite[Exercise 4.7.20]{Dix}, it suffices to show that, for any
$a \in \A_+$, any non-zero point of the spectrum of $a$ is an isolated point.
Suppose, for the sake of contradiction, that there exists $a \in \A_+$ whose
spectrum contains a strictly positive non-isolated point $\alpha$. In other
words, for every $\delta > 0$,
$((\alpha - \delta, \alpha) \cup (\alpha, \alpha + \delta))
 \cap \sigma(a) \neq \emptyset$.
Without loss of generality, we can assume $0 \leq a \leq 1$.
Thus, we can find countably many mutually disjoint non-empty subsets $S_i$ of
$(\alpha/2,\infty) \cap \sigma(a)$. Denote the corresponding
spectral projections by $p_i$ (that is, $p_i = \chi_{S_i}(a)$).
These projections belong to $\A^{\star\star}$.
Furthermore, $p_i \leq (\inf S_i)^{-1} a$,
hence, by the hereditary property, these projections belong to $\A$.

Now consider the linear map $T : \A \to \A : x \mapsto axa$. Then $T^{\star\star}$
is also implemented by $x \mapsto axa$. If $0 \leq x \leq \one$, then
$axa \leq a^2$, hence $axa \in \A$. Therefore,
$T^{\star\star}$ takes $\A^{\star\star}$ to $\A$. By Gantmacher's Theorem
(see e.g. \cite[Theorem 5.23]{AB}), $T$ is weakly compact.
However, $T$ is an isomorphism on the copy of
$c_0$, spanned by the projections $p_i$, leading to a contradiction.
\end{proof}

\begin{remark}\label{rem:blecher}
The above result was independently proved in \cite{Ble}, using
a different method.
\end{remark}

\subsection{Positive Schur Property. Compactness of order intervals
 in Schatten spaces} \label{ss:psp}

An OBS $X$ is said to have the \emph{Positive Schur Property} (\emph{PSP})
if every weakly null positive sequence is norm convergent to 0 and $X$ has the
\emph{Super Positive Schur Property} (\emph{SPSP}) if every positive weakly convergent
sequence is norm convergent. Clearly, the Schur Property implies the SPSP,
which, in turn, implies the PSP. Note that, if $X$ has
the SPSP, then, by the Eberlein-Smulian Theorem, any weakly compact subset of $X_+$
is compact.

The PSP and SPSP of Banach lattices have been investigated earlier.
By \cite{Wn89a}, the Schur Property and the PSP coincide for atomic Banach lattices.
In \cite{KM00}, it is shown that $\ell_1$ is the only symmetric sequence space
with the Schur Property (by Remark \ref{rem:not l_1} below, the symmetry assumption
is essential). \cite{KM02} gives a criterion for the PSP of Orlicz spaces.

\begin{lemma} \label{lem:pos_seq}
Suppose $\ce$ is a symmetric sequence space, and
$(A_n)$ is a positive bounded sequence in $\se$ without a convergent
subsequence. Then there exist a subsequence $(A_{n_k})$ and $c>0$ such that
$\|R_kA_{n_k}\|>c$ for every $k$.
\end{lemma}

\begin{proof} Assume there is no such subsequence, that is
$$\lim_m {\sup_n{\|R_mA_n\|}}=0.$$
Applying Lemma~\ref{lem:maps from S_E} when $T$ is the identity operator, 
we obtain the inequality
\begin{eqnarray*}
\|A_n-Q_mA_n\|      &\le & \|A_n- Q_mA_n-R_mA_n\|+\|R_mA_n\|\\
          & \le & 2 \|Q_mA_n\|^{\frac{1}{2}}
                \|R_mA_n\|^{\frac{1}{2}}+ \|R_mA_n\| .
\end{eqnarray*}
Thus, $\lim_m \sup_n \|A_n-Q_mA_n\| = 0$. However, $Q_m$ is a finite rank map, hence
the set $(A_n)$ is relatively compact, a contradiction.
\end{proof}

\begin{proposition} \label{prop:glid_hump}
Suppose $\ce$ is a separable symmetric sequence space.
Let $(A_n)$ be a weakly null positive sequence in $\se(H)$,
which contains no convergent subsequences. Then there exists $c > 0$
with the property that, for any $\vr \in (0,1)$, there exist sequences
$1 = n_1 < n_2 < \ldots$ and $0 = m_0 < m_1 < \ldots$, so that
$\inf_k \|A_{n_k}\| > c$, and
$$
\sum_k \|A_{n_k} - (P_{m_k} - P_{m_{k-1}}) A_{n_k} (P_{m_k} - P_{m_{k-1}})\| < \vr .
$$
Consequently, the sequence $(A_{n_k})$ is equivalent to a disjoint sequence of
positive finite dimensional operators.
\end{proposition}

\begin{proof}
By the separability (equivalently, order continuity) of $\ce$,
there exists a projection $p \in B(H)$ with separable range, so
that $p A_k p = A_k$ for any $k$. Thus, it suffices to prove
our proposition in $\se$.

Furthermore, the order continuity of $\ce$ implies that the finite rank
operators are dense in $\se$. It is easy to see that, for any rank $1$
operator $u$, $\lim_n \|u - Q_n u\| = 0$. Thus,
$\lim_n \|x - Q_n x\| = 0$ for any $x \in \ce$.

By scaling, we can assume $\sup_n \|A_n\| = 1$.
Applying Lemma~\ref{lem:pos_seq}, and passing to a subsequence if necessary,
we may assume  that $\|R_nA_{n}\|>c$, for some positive number $c$.
We construct the sequences $(n_k)$ and $(m_k)$ recursively. Set $n_1 = 1$
and $m_0 = 0$. As noted above, there exists $m_1 > m_0$ so that
$\|A_{n_1} - P_{m_1} A_{n_1} P_{m_1}\| < \vr/2$.

Suppose we have already selected $0 = m_0 < m_1 < \ldots < m_j$ and
$1 = n_1 < n_2 < \ldots < n_j$ so that, for $1 \leq j \leq k$,
$$
\|A_{n_k} - (P_{m_k} - P_{m_{k-1}}) A_{n_k} (P_{m_k} - P_{m_{k-1}})\| <
 2^{-j} \vr .
$$
As $Q_m$ is a finite rank operator for any $m$, and the sequence $(A_n)$
is weakly null, $\lim_n \|Q_m A_n\| = 0$. Consequently, there exists
$n_{k+1} > n_k$ so that $\|Q_{m_k} A_{n_{k+1}}\| < 2^{-2(k+1)-4} \vr^2$.
Then
\begin{eqnarray*}
&
\|A_{n_{k+1}} - R_{m_k} A_{n_{k+1}}\| \leq
\|A_{n_{k+1}} - R_{m_k} A_{n_{k+1}} - Q_{m_k} A_{n_{k+1}}\|
 + \|Q_{m_k} A_{n_{k+1}}\|
\\
&
\leq
2 \|Q_{m_k} A_{n_{k+1}}\|^{1/2} \|R_{m_k} A_{n_{k+1}}\|^{1/2}
 + \|Q_{m_k} A_{n_{k+1}}\|
< 2^{-(k+2)} \vr .
\end{eqnarray*}
Now find $m_{k+1}$ so that
$\|R_{m_k} A_{n_{k+1}} - Q_{m_{k+1}} R_{m_k} A_{n_{k+1}}\| < 2^{-(k+2)} \vr$.
\end{proof}

\begin{proposition} \label{prop:S_1_SPSP}
For any Hilbert space $H$, $\is_1(H)$ has the SPSP.
\end{proposition}

\begin{proof}
It suffices to consider the case of infinite dimensional $H$.
Suppose $A_0, A_1, A_2, \ldots$ are positive elements of $\is_1(H)$,
and $A_n \to A_0$ weakly. Then there exist projections $p_0, p_1, p_2, \ldots$
with separable range, so that $p_i A_i p_i = A_i$ for every $i$.
Then $p = \vee_{i \geq 0} p_i$ has separable range, and $p A_i p = A_i$
for every $i$. Thus, we can assume that $H = \ell_2$.

By Lemma~\ref{lem:pos_seq} there exist $c>0$ and a
subsequence such that $\|R_kA_{n_k}\|>c$. Since $R_m \ge R_k$ when $m \le k$,
we have $\trace (R_mA_{n_k}) >c$ for every $k$. On the other hand we can always
pick $m$ such that $\trace(R_mA)=\|R_mA\| <c$.
This contradicts $A_n \to A$  weakly.
\end{proof}




\begin{proposition}\label{prop:se_PSP}
Suppose $\ce$ is a strongly symmetric sequence space, and $H$
is an infinite dimensional Banach space.
Then the following are equivalent:
\begin{enumerate}
\item $\ce=\ell_1$.
\item $\ce$ has the Schur Property.
\item $\ce$ has the PSP.
\item $\ce$ has the SPSP.
\item $\se(H)$ has the PSP.
\item $\se(H)$ has the SPSP.
\end{enumerate}
\end{proposition}

\begin{proof}
$(1) \Rightarrow (2)$ is well known. The implications
$(2) \Rightarrow (4) \Rightarrow (3)$, $(6) \Rightarrow (4)$,
and $(6) \Rightarrow (5) \Rightarrow (3)$ are obvious.
$(1) \Rightarrow (6)$ follows from Proposition~\ref{prop:S_1_SPSP}.

$(3) \Rightarrow (1)$. Assume that basis $(e_n)$ of $\ce$ is not equivalent
to the canonical basis of $\ell_1$. By symmetry, $(e_n)$ contains
no subsequence equivalent to the canonical basis of $\ell_1$. By
Rosenthal's dichotomy, the sequence $(e_n)$ is weakly null,
which contradicts the PSP.
\end{proof}

We complete this section by (partially) describing Banach lattices possessing various
versions of the Schur Property.

\begin{proposition}\label{prop:SPSP_atomless}
Any Banach lattice $E$ with the SPSP is atomic.
\end{proposition}

Recall that a Banach lattice is called \emph{atomic} if it is the
band generated by its atoms.

\begin{proof}
Clearly, a Banach lattice with the SPSP cannot contain a lattice
copy of $c_0$. Theorems 2.4.12 and 2.5.6 of \cite{M-N} show that
$E$ is a KB-space. In particular, $E$ is order continuous.
By \cite[Proposition~1.a.9]{LT2}, without loss of generality, we may assume $E$ is atomless and  has a weak unit. Therefore, by  \cite[Theorem~1.b.4]{LT2}, there exists an atomless probability measure space $(\Omega,\mu)$, so that
$L_\infty(\mu) \subset E \subset L_1(\mu)$. Suppose,
furthermore, that $e \in E_+ \backslash \{0\}$. Find $S \subset \Omega$
of finite measure, so that $e\chi_{S} > \alpha \chi_{S} $ for some positive number $\alpha$. By the proof of
\cite[Proposition 2.1]{CW}, there exists a weakly null sequence $(f_n)$,
so that $|f_n| = 1$ $\mu$-a.e. on $S$, $f_n = 0$ on $\Omega \backslash S$,
and $f_n \to 0$ in $\sigma(L_\infty(\mu), L_1(\mu))$. Letting
$e_n = e + f_n$, we conclude that $e_n \geq 0$ for every $n$,
and $e_n \to e$ weakly, but not in norm.
\end{proof}




\begin{proposition}\label{prop:equiv_prop}
For any order continuous Banach lattice $E$ the SPSP, the PSP, and
the Schur Property are equivalent.
\end{proposition}

\begin{proof}
 Proposition~\ref{prop:SPSP_atomless} implies $E$ is atomic. Therefore
the result follows from the fact that the lattice operations are weakly
sequentially continuous, see  \cite[Proposition~2.5.23]{M-N}.
\end{proof}

\begin{remark}\label{rem:not l_1}
An order continuous atomic Banach lattice with the Schur Property need not
be isomorphic to $\ell_1$, even as a Banach space. Indeed, suppose
$(E_n)$ is a sequence of finite dimensional lattices. Then
$E = (\sum_{n=1}^\infty E_n)_{\ell_1}$ has the Schur Property.
If, for instance, $E_n = \ell_2^n$, $E$ is not isomorphic to $\ell_1$.
We do not know of any Banach lattice with the Schur Property
which is not isomorphic to an $\ell_1$ sum of finite dimensional spaces.
\end{remark}

\subsection{Compactness of order intervals in preduals of von Neumann algebras}
\label{ss:psp predual}

Following \cite[Definition III.5.9]{Tak},
we say that a von Neumann algebra $\A$ is \emph{atomic} if every
projection in $\A$ has a minimal subprojection.
Note that $\A$ is atomic iff it is isomorphic to
$(\sum_{i \in I} B(H_i))_{\ell_\infty(I)}$, for some index
set $I$, and collection of Hilbert spaces $(H_i)_{i \in I}$.
Indeed, any von Neumann algebra of the above form is atomic.
To prove the converse, note that an atomic algebra must be
of type $I$. Moreover, it can be written as
$\A = (\sum_{j \in J} \A_j)_{\ell_\infty(J)}$,
where $\A_j$ is an atomic algebra of type $I_j$.
By \cite[Theorem V.1.27]{Tak} (see also \cite[Theorem 6.6.5]{KR2}
and \cite[III.1.5.3]{Bl}),
$\A_j$ is isomorphic to ${\mathcal{C}}_j \overline{\otimes} B(H_j)$,
where ${\mathcal{C}}_j$ is the center of $\A_j$.
Denote the set of all minimal projections in ${\mathcal{C}}_j$
by $F_j$. Then the elements of $F_j$ are mutually orthogonal, and
their join equals the identity of ${\mathcal{C}}_j$. Thus,
${\mathcal{C}}_j$ is isomorphic to $\ell_\infty(F_j)$.
Alternatively, one could use \cite[III.1.5.18]{Bl}
and its proof to show that ${\mathcal{C}}_j$ is an
$\ell_\infty$ space.

\begin{theorem}\label{thm:ord int}
For a von Neumann algebra $\A$, the following are equivalent:
\begin{enumerate}
\item
$\A$ is an atomic von Neumann algebra.
\item
$\A_\star$ has the SPSP.
\item
All order intervals in $\A_\star$ are compact.
\end{enumerate}
\end{theorem}

\begin{remark}\label{rem:A_*PSP}
Note that the predual of any von Neumann algebra has the PSP.
Indeed, suppose $(f_n)$ is a sequence of positive elements of $A_\star$,
converging weakly to $0$. Then $\|f_n\| = \langle f_n, \one \rangle$,
hence $\lim_n \|f_n\| = \lim_n \langle f_n, \one \rangle = 0$.

Also, any order interval in the predual of a von Neumann algebra is
weakly compact. Indeed, 
suppose $f$ is a positive element of $\A_\star$. Then $[0,f]$
is convex and closed. For any $g \in [0,f]$ and $a \in \A$,
Cauchy-Schwarz Inequality \cite[Proposition I.9.5]{Tak} yields
$|g(a)|^2 \leq g(\one) g(a^* a) \leq f(\one) f(a^* a)$.
By \cite[Theorem III.5.4]{Tak}, $[0,f]$ is relatively weakly compact.
\end{remark}

To prove Theorem \ref{thm:ord int}, we need to determine when $\A_\star$ contains
an order copy of $L_1(0,1)$, complemented via a positive projection.

\begin{proposition}\label{prop:compl L1}
For a von Neumann algebra $\A$, the following statements hold:
\begin{enumerate}
\item 
If $\A$ is atomic, then $\A_\star$ does not contain $L_1(0,1)$ isomorphically.
\item
If $\A$ is not atomic, then there exists an isometric order isometry
$j : L_1(0,1) \to \A_\star$, and a positive projection
$P : \A_\star \to \ran(j)$.
\end{enumerate}
\end{proposition}

\begin{proof}
$(1)$
Note that, for any Hilbert space $H$, $\is_1(H)$ does not contain $L_1(0,1)$
isomorphically. Indeed, otherwise, by the separability argument, we would
be able to embed $L_1(0,1)$ into $\is_1$. This, however, is impossible, by e.g.
\cite{Fr}. To finish the proof of $(1)$, recall that, if $\A$ is atomic,
then it can be identified with $(\sum_i B(H_i))_\infty$, and
$\A_\star$ is isometric to $(\sum_i \is_1(H_i))_1$.

$(2)$
We can write $\A = \A_I \oplus \A_{\lnot I}$, where $\A_I$ has type $I$, and
$\A_{\lnot I}$ has no type $I$ components (that is, it is a direct sums
of von Neumann algebras of types $II$ and $III$). Either $\A_I$ is not atomic,
or $\A_{\lnot I}$ is non-trivial.

If $\A_I$ is not an atomic von Neumann algebra,
write $\A_I = (\sum_{s \in S} \A_s)_{\ell_\infty(S)}$, with
$\A_s = {\mathcal{C}}_s \overline{\otimes} B(H_s)$
(${\mathcal{C}}_s$ is the center of $\A_s$).
By \cite[Theorem III.1.18]{Tak}, ${\mathcal{C}}_s$
is isomorphic to $L_\infty(\nu_s)$, for some locally
finite measure $\nu_s$. Consequently,
$\A_\star$ contains $L_1(\nu_s) \otimes \is_1(H_s)$ as a positively
and completely contractively complemented subspace.
As $\A_I$ is not an atomic von Neumann algebra,
then $\nu_s$ is not a purely atomic measure, for some $s$.
By the above,
$\A_\star$ contains $L_1(\nu_s) \otimes \is_1(H_s)$ as a positively
and completely contractively complemented subspace.
Furthermore, $L_1(\nu_s)$ is complemented in
$L_1(\nu_s) \otimes \is_1(H_s)$ via a positive projection $Q$: just
pick a rank one projection $e \in B(H_s)$, and set
$Q(x) = (I_{L_1(\nu_s)} \otimes e)x(I_{L_1(\nu_s)} \otimes e)$.
Finally, $L_1(\nu_s)$ contains a positively complemented copy of $L_1(0,1)$.
Indeed, we can represent $L_1(\nu_s)$ a direct sum of spaces $L_1(\sigma_i)$,
where $\sigma_i$ is a finite measure. Since $\nu_s$ is not purely atomic,
the same is true for $L_1(\sigma_i)$, for some $i$.
By \cite[Theorem III.1.22]{Tak} (or \cite[Theorem 9.4.1]{KR2}),
$L_1(\nu_s)$ contains a positively complemented copy of $L_1(0,1)$.

Now suppose $\A_{\lnot I}$ is non-trivial. By the reasoning of
\cite[Page 217]{MNha}, $\A_{\lnot I}$ contains a von Neumann subalgebra
${\mathcal{B}}$, isomorphic to the hyperfinite $II_1$ factor ${\mathcal{R}}$.
Furthermore, there exists a normal contractive projection (conditional expectation)
$P : \A_{\lnot I} \to {\mathcal{B}}$. By \cite[Theorem III.3.4]{Tak}, $P$ is positive.
Consequently, $\A_\star$ contains a copy of ${\mathcal{R}}_\star$, complemented via
a positive contractive projection.

Let $\mu$ be the ``canonical'' measure on the Cantor set $\Delta$, defined
as follows: represent $\Delta = \{0,1\}^\N$, and write $\mu = \nu^\N$,
where the measure $\nu$ on $\{0,1\}$ satisfies $\nu(0) = \nu(1) = 1/2$.
For $\alpha = (i_1, \ldots, i_n) \in I = \{0,1\}^{< \N}$, define the function
$f_\alpha$ by setting 
$f_\alpha(j_1, j_2, \ldots) = \prod_{k=1}^n \delta_{i_k,j_k}$
(here, $\delta_{i,j}$ stands for Kronecker's delta).
Note that $f_\alpha$ and $f_\beta$ have disjoint supports if $\alpha$
and $\beta$ are different bit strings of the same length.
Moreover, $f_\alpha = f_{(\alpha,0)} + f_{(\alpha,1)}$.
Clearly, $L_1(\mu)$ is the closed linear span of the functions $f_\alpha$.
Subdividing $(0,1)$ appropriately, one can also construct an isometric
order isomorphism between $L_1(\mu)$ and $L_1(0,1)$.

It therefore suffices to show that there exists an order isometry
$J : L_1(\mu) \to {\mathcal{R}}_\star$, so that the range of $J$
is the range of a positive projection.
To prove this, let $\Delta_n = \{0,1\}^n$, and denote by $\mu_n$ the
product of $n$ copies of $\nu$. In this notation, $L_1(\mu_n)$ is
isometric to $\ell_1^{2^n}$. We can also identify $L_1(\mu_n)$ with
$\span[f_\alpha : |\alpha| = n]$. Let
$i_n$ be the formal identity $L_1(\mu_{n-1}) \to L_1(\mu_n)$
(taking $f_\alpha$ to itself, when $|\alpha| \leq n$).

For $n \in \N$, consider the map
$j_n : M_{2^{n-1}} \to M_{2^n} : x \mapsto x \otimes M_2$.
Denote by $\tr_n$ the normalized trace on $M_{2^n}$, and
by $M_{2^n}^\star$ the dual of $M_{2^n}$ defined using $\tr_n$.
Then $j_n : M_{2^{n-1}}^\star \to M_{2^n}^\star$ is an isometry.
Furthermore, the diagonal embedding
$u_n : L_1(\mu_n) \to M_{2^n}^\star$ is an isometry, and
$u_n i_n = j_n u_{n-1}$. We can view both
$M_{2^{n-1}}^\star$ and $L_1(\mu_n)$ as subspaces of $M_{2^n}^\star$,
Furthermore, for any $n$ there exist positive contractive
unital projections $p_n : M_{2^n}^\star \to L_1(\mu_n)$ and
$q_n : M_{2^n}^\star \to M_{2^{n-1}}^\star$ (the ``diagonal''
and ``averaging'' projections, respectively). We then have
$p_n j_n = i_n p_{n-1}$.

It is well known (see e.g. \cite[Theorem 3.4]{PiLP}) that
${\mathcal{R}}_\star$ can be viewed as $\overline{\cup_n M_{2^n}^\star}$.
Moreover, for any $n$ there exists a positive contractive
unital projection $\tilde{q}_n : {\mathcal{R}}_\star \to M_{2^n}^\star$
(with $\tilde{q}_n|_{M_{2^n}^\star} = q_{n+1} \ldots q_N$). 
Now identify $L_1(\mu)$ with $\overline{\cup_n L_1(\mu_n)}$,
and define the projection $P : {\mathcal{R}}_\star \to J(L_1(\mu))$
by setting $P|_{M_{2^n}^\star} = q_n$.
\end{proof}

\begin{proof}[Proof of Theorem \ref{thm:ord int}]
If (1) holds, then $\A = (\sum_i B(H_i))_\infty$, hence
$\A_\star = (\sum_i \is_1(H_i))_1$.
(2) and (3) follow from Propositions \ref{prop:se_PSP} and
\ref{lem:psol comp}, respectively.

Now suppose $\A$ is not atomic. By Proposition \ref{prop:compl L1},
$\A_\star$ contains a (positively and contractively complemented) lattice copy of $L_1(0,1)$.
To finish the proof, note that $L_1(0,1)$ fails the
SPSP, and has non-compact order intervals. Indeed, let $f = \one$, and
$f_n = \one + r_n$, where $r_1, r_2, \ldots$ are Rademacher functions.
Then $f_n \to f$ weakly, but not in norm. This witnesses the failure of
the SPSP. Moreover, $f_n/2 \in [0,\one]$, hence the order interval
$[0,\one]$ is not compact.
\end{proof}


\section{Main results on majorization}\label{sec:main}

\subsection{Compact operators on non-commutative function spaces}\label{ss:schatten_main}

First we consider maps from ordered Banach spaces into Schatten spaces.

\begin{proposition}\label{prop:compact2}
Suppose $\ce$ is a separable symmetric sequence space, $H$ is a Hilbert space,
$A$ is a generating OBS, and $0 \le T \le S: A \to \se(H)$
(not necessary linear).  If $S$ is compact, then $T$ is compact.
\end{proposition}

\begin{proof}
It is enough to show $T(\ball(A)_+)$ is relatively compact. Thus follows from
Lemma \ref{lem:psol comp}, since $T(\ball(A)_+) \subseteq \psol(S(\ball(A)_+))$.
\end{proof}

For operators into Schatten spaces, we have:

\begin{proposition}\label{prop:compact}
Suppose $\ce$ is a separable symmetric sequence space,
and $H$ is a Hilbert space.

$(1)$
If $\ce$ does not contain $\ell_1$, and operators $T$ and $S$ from $\se(H)$
to a normal OBS $Z$ satisfy $0 \leq T \leq S$, then the compactness of
$S^\star$ implies the compactness of $T^\star$.

$(2)$
Conversely, suppose $\ce$ contains $\ell_1$,
and a Banach lattice $Z$ is either not atomic, or not order continuous.
Then there exist $0 \leq T \leq S : \se(H) \to Z$ so that
$S$ is compact, but $T$ is not.
\end{proposition}

\begin{proof}
$(1)$
By \cite[Theorem 1.c.9]{LT1}, $\ce^\star$ is separable.
Now apply Proposition \ref{prop:compact2}.

$(2)$
By \cite{Wi},
there exist $0 \leq \tilde{T} \leq \tilde{S} : \ell_1 \to Z$
so that $\tilde{S}$ is compact, but $\tilde{T}$ is not. By
Proposition \ref{prop:l_1_S_E_cont}, there exists a lattice
isomorphism $j : \ell_1 \to \se$, and a positive projection $P$
from $\se$ onto $j(\ell_1)$. Then the operators $T = \tilde{T} j^{-1} P$
and $S = \tilde{S} j^{-1} P$ have the desired properties.
\end{proof}


Finally we deal with operators on general non-commutative function spaces.

\begin{proposition}\label{prop:no OC*}
Suppose $\ce$ is a strongly symmetric non-commutative function space, such that
$\ce^\times$ is not order continuous. Suppose, furthermore, that a
symmetric non-commutative function space ${\mathcal{F}}$ contains
non-compact order intervals. Then there exist
$0 \leq T \leq S : \ce \to {\mathcal{F}}$, so that
$S$ has rank $1$, and $T$ is not compact.
\end{proposition}

Note that many spaces ${\mathcal{F}}$ contain non-compact order
ideals. Suppose, for instance, that ${\mathcal{F}}$ arises from a von
Neumann algebra $\A$ that is not atomic, and is equipped with a normal
faithful semifinite trace $\tau$. Using the type decomposition,
we can find a projection $p \in \A$ with a finite trace.
Then the interval $[0,p]$ is not compact. Indeed,
\cite[Proposition V.1.35]{Tak} allows us to construct a family
of projections $(p_{ni})$ ($n \in \N$, $1 \leq i \leq 2^n$), so
that (i) $p = p_{11} + p_{12}$, and
$p_{ni} = p_{n+1,2i-1} + p_{n+1,2i}$ for any $n$ and $i$, and
(ii) all projections $p_{ni}$ are equivalent.
Then the family $q_n = \sum_{i=1}^{2^{n-1}} p_{n,2i}$
is a sequence in $[0,p]$, with no convergent subsequences.

Note that, for fully symmetric non-commutative sequence spaces,
order continuity is fully described by Corollary \ref{cor:crit OC}.

\begin{lemma}\label{lem:dual OC}
Suppose $\ce$ is a strongly symmetric non-commutative function space,
so that $\ce^\times$ is not order continuous. Then there exists
an isomorphism $j : \ell_1 \to \ce$, so that both $j$ and $j^{-1}$
are positive, and $j(\ell_1)$ is the range of a positive projection.
\end{lemma}

\begin{proof}
By \cite{DDdP93}, $\ce^\times$ is fully symmetric. By
Lemma \ref{lem:+l_infty}, there exists $x \in \ball(\ce^\times)_+$,
and a sequence of mutually orthogonal projections $(e_i)$, so
that $(\alpha_i) \mapsto \sum \alpha_i e_i x e_i$ determines
a positive embedding of $\ell_\infty$ into $\ce^\times$.
For each $i$, find $y_i \in \ce_+$ so that $e_i y_i e_i = y_i$,
$\|y_i\| < 2 \|e_i x e_i\|^{-1}$, and $\langle e_i x e_i, y_i \rangle = 1$.
The map $j : \ell_1 \to \ce : (\alpha_i) \mapsto \sum_i \alpha_i y_i$
determines a positive isomorphism. Furthermore, define
$U : \ce \to \ell_1 : y \mapsto (\langle e_i x e_i , y \rangle)_i$.
Clearly, $U$ is a bounded positive map, and $Uj = I_{\ell_1}$.
Therefore, $jU$ is a positive projection onto $j(\ell_1)$.
\end{proof}

\begin{proof}[Proof of Proposition \ref{prop:no OC*}]
In view of Lemma \ref{lem:dual OC}, it suffices to construct
$0 \leq T \leq S : \ell_1 \to {\mathcal{F}}$, so that $S$
has rank $1$, and $T$ is not compact.
Pick $y \in {\mathcal{F}}$, so that $[0,y]$ is not compact.
Then find a sequence $(y_i) \subset [0, y]$, without convergent subsequences.
Denote the canonical basis of $\ell_1$ by $(\delta_i)$. Let
$\delta_i^\star$ be the biorthogonal functionals in $\ell_\infty$.
Following \cite{Wi}, define $S$ and $T$ by setting $S \delta_i = y$, and
$T \delta_i = y_i$. In other words, for $a = (\alpha_i) \in \ell_1$,
$S a = \langle \one, a \rangle y$, and
$T a = \sum_i \langle \delta_i^\star, a \rangle y_i$.
It is easy to see that $\rank S = 1$, and $0 \leq T \leq S$.
Moreover, $T(\ball(\ell_1))$ contains the non-compact set
$\{y_1, y_2, \ldots\}$, hence $T$ is not compact.
\end{proof}

\subsection{Compact operators on $C^*$-algebras and their duals}\label{ss:C*alg}
In this section, we determine the $C^*$-algebras $\A$ with the property that
every operator on $\A$, dominated by a compact operator, is itself compact.
First we introduce some definitions. Let $\A$ be a $C^*$-algebra, and
consider $f \in \A^\star$. Let $e \in \A^{\star\star}$ be its support
projection. Following \cite{Je}, we call $f$ \emph{atomic} if every non-zero
projection $e_1 \leq e$ dominates a minimal projection (all projections are
assumed to ``live'' in the enveloping algebra $\A^{\star\star}$).
Equivalently, $f$ is a sum of pure positive functionals. We say that $\A$
is \emph{scattered} if every positive functional is atomic. By \cite{Hu}, \cite{Je},
the following three statements are equivalent: (i) $\A$ is scattered;
(ii) $\A^{\star\star} = (\sum_{i \in I} B(H_i))_\infty$; (iii) the
spectrum of any self-adjoint element of $\A$ is countable. Consequently
(see \cite[Exercise 4.7.20]{Dix}), any compact $C^*$-algebra is scattered. 
In \cite{Woj}, it is proven that a separable $C^*$-algebra
has separable dual if and only if it is scattered.

The main result of this section is:

\begin{theorem}\label{thm:scatter}
Suppose $\A$ and ${\mathcal{B}}$ are $C^*$-algebras,
and $E$ is a generating OBS.
\begin{enumerate}
\item
Suppose $\A$ is a scattered.
Then, for any $0 \leq T \leq S : E \to \A^\star$, the compactness of
$S$ implies the compactness of $T$.
\item
Suppose ${\mathcal{B}}$ is a compact.
Then, for any $0 \leq T \leq S : E \to {\mathcal{B}}$, the compactness of
$S$ implies the compactness of $T$.
\item
Suppose $\A$ is not scattered, and ${\mathcal{B}}$ is not compact.
Then there exist $0 \leq T \leq S : \A \to {\mathcal{B}}$,
so that $S$ has rank $1$, while $T$ is not compact.
\end{enumerate}
\end{theorem}

From this, we immediately obtain:

\begin{corollary}\label{cor:scat}
Suppose $\A$ and ${\mathcal{B}}$ are $C^*$-algebras.
Then the following are equivalent:
\begin{enumerate}
\item
At least one of the two conditions holds:
(i) $\A$ is scattered, (ii) ${\mathcal{B}}$ is compact.
\item
If $0 \leq T \leq S : \A \to {\mathcal{B}}$, and $S$ is compact,
then $T$ is compact.
\end{enumerate}
\end{corollary}

It is easy to see that a von Neumann algebra is scattered
if an only if it is finite dimensional if and only if it
is compact. This leads to:

\begin{corollary}\label{cor:scat_vna}
If von Neumann algebra $\A$ and ${\mathcal{B}}$ are infinite dimensional,
then there exist $0 \leq T \leq S : \A \to {\mathcal{B}}$,
so that $S$ has rank $1$, while $T$ is not compact.
\end{corollary}

We establish similar results about preduals of von Neumann algebras.

\begin{lemma}\label{lem:C*alg_comp_int}
(1)
Suppose $\A$ is an atomic von Neumann algebra, and $E$ is a generating OBS.
Then $0 \leq T \leq S: E \to \A_\star$, where $S$ is a compact operator,
implies $T$ is compact.

(2)
Suppose $\A$ is a von Neumann algebra, and $\A_I, \A_{II}, \A_{III}$ are
its summands of type $I$, $II$, and $III$, respectively. Suppose,
furthermore, that one of the three statements is true:
$(i)$ $\A_I$ is not atomic, $(ii)$ $\A_{II}$ is not empty,
$(iii)$ $\A_{III}$ is non-empty, and has separable predual.
Then there exists $0 \leq T \leq S: \A_\star \to \A_\star$,
so that $S$ is compact, and $T$ is not.
\end{lemma}

\begin{proof}
(1) The weak compactness of $S$ implies, by Theorem~\ref{th:weak_comp},
the weak compactness of $T$. By Theorem~\ref{thm:ord int}, $\A_\star$
has the SPSP, hence $T(\ball(E)_+)$ is relatively compact. Thus,
$T(\ball(E))$ is relatively compact as well, hence $T$ is compact.

(2) It suffices to show that there exists an order isomorphism
$j: L_1(0,1) \to \A_\star$, so that there exists a positive
projection $P$ onto $\ran(j)$. Indeed, by \cite{Wi}, there exist
operators $0 \leq T_0 \leq S_0 : L_1(0,1) \to L_1(0,1)$, so that
$S_0$ is compact, and $T_0$ is not. Then $T = j T_0 j^{-1} P$
and $S = j S_0 j^{-1} P$ have the desired properties.
The existence of $j$ and $P$ as above follows from the proof of
Proposition \ref{prop:compl L1}.
\end{proof}

To establish Theorem \ref{thm:scatter}, we need a series of lemmas.

\begin{lemma}\label{lem:not_comp}
Suppose $\A$ is a $C^*$-algebra for which $\A^\star$ has non-compact
order intervals, and a Banach lattice $E$ is not order continuous.
Then there exist $0 \leq T \leq S : \A \to E$,
so that $S$ has rank $1$, while $T$ is not compact.
\end{lemma}

\begin{proof}
By \cite[Theorem 2.4.2]{M-N},
there exists $y \in E_+$, and normalized elements
$y_1, y_2, \ldots \in [0,y]$ with disjoint supports.
By our assumption there exist $\psi \in \A^\star_{+}$ and a sequence
$(\phi_i) \subset [0, \psi]$ which does not have convergent subsequences.
By Alaoglu\rq{s} theorem we may assume $\phi_i \to \phi$ in weak$^*$ topology.
Define  two operators via
\begin{displaymath}
Sx=\psi(x)y \text{ and } Tx=\phi(x)y+\sum_{n=1}^{\infty}{(\phi_n-\phi)(x)y_n}.
\end{displaymath}
Note that $T$ is well defined: $(\phi_n -\phi)(x) \to 0$ for all $x$, hence
\begin{displaymath}
\norm{\sum_{n=m+1}^{k}{(\phi_n-\phi)(x)y_n}} \leq
 \sup_{m > n}|(\phi_m-\phi)(x)|\norm{y}
{\underset{n \to \infty}{\longrightarrow}} 0 .
\end{displaymath}
Moreover, for  any $x>0$ and $N \in \mathbb{N}$ we have
\begin{equation*}
 \phi(x)y+\sum_{n=1}^{N}{(\phi_n-\phi)(x)y_n}=
\phi(x)(y-\sum_{n=1}^{N}{y_n})+\sum_{n=1}^{N}{\phi_n(x)} y_n \ge 0 ,
\end{equation*}
and
\begin{eqnarray*}
  & \psi(x) y - \phi(x) y - \sum_{n=1}^N (\phi_n-\phi)(x)y_n = \\
  & \psi(x) y - \sum_{n=1}^n \phi_n(x) y_n -
     \phi(x) \big( y - \sum_{n=1}^N y_n \big) \geq \\
  & \big( \psi(x) - \phi(x) \big) \big( y - \sum_{n=1}^N y_n \big) .
\end{eqnarray*}
By sending $N$ to infinity, we obtain that $0 \le Tx \le Sx$ for every $x>0$.
Clearly, $\rank S = 1$. It remains to show that $T^\star$ is not compact.
Note that there exist norm one $f_1, f_2, \ldots \in E^\star$ so that
$f_n(y_m) = \delta_{nm}$.
It is easy to see that
$T^\star f = f(y) \phi + \sum_{n=1}^\infty f(y_n) (\phi_n - \phi)$, hence
$T^\star  f_m = (f_m(y) - 1) \phi + \phi_m$. The sequence
$(T^\star f_m)$ has no
convergent subsequences, since if it had, $(\phi_m)$ would have
a convergent subsequence, too. This rules out the compactness of $T^\star$.
\end{proof}

\begin{corollary}\label{cor:not_comp}
Suppose a $C^*$-algebra ${\mathcal{B}}$ is not compact,
and $\A^\star$ has non-compact order intervals. Then there exist
$0 \leq T \leq S : \A \to {\mathcal{B}}$,
so that $S$ has rank $1$, while $T$ is not compact.
\end{corollary}

\begin{proof}
By Lemma \ref{lem:not_comp}, it suffices to show that ${\mathcal{B}}$
contains a Banach lattice which is not order continuous.
By \cite[Exercise 4.7.20]{Dix}, ${\mathcal{B}}$ contains a positive
element $b$, whose spectrum contains a positive non-isolated point.
Then the abelian $C^*$-algebra ${\mathcal{B}}_0$, generated by $b$, is not order
continuous. Indeed, suppose $\alpha > 0$ is not an isolated point of
$\sigma(a)$. Then there exist disjoint subintervals
$I_i = (\beta_i,\gamma_i) \subset (\alpha/2, 3\alpha/2)$, so that
$\delta_i = (\beta_i + \gamma_i)/2 \in \sigma(b)$ for every $i \in \N$.
For each $i$, consider the function $\sigma_i$, so that
$\sigma_i(\beta_i) = \sigma_i(\gamma_i) = 0$,
$\sigma_i((\beta_i+\gamma_i)/2) = 1$, and $\sigma_i$
is defined by linearity elsewhere. Then the elements
$y_i = \sigma_i(b)$ belongs to ${\mathcal{B}}_0$, are
disjoint and normalized, and $y_i \leq y = 2 \alpha^{-1} b$.
\end{proof}

\begin{proof}[Proof of Theorem \ref{thm:scatter}]
(1) If $\A$ is scattered, then $\A^{\star\star}$ 
is atomic.
Now invoke Lemma \ref{lem:C*alg_comp_int}(1).

(2) By assumption, $M = S(\ball(E)_+)$ is relatively compact,
and $T(\ball(E)_+) \subset \psol(M)$. By Proposition \ref{prop:dual_alg},
$T(\ball(E)_+)$ is relatively compact.

(3) Combine Theorem \ref{thm:ord int} with
Corollary \ref{cor:not_comp}.
%
\end{proof}

\subsection{Comparisons with multiplication operators}\label{ss:mult_ops}

Suppose $\A$ is a $C^*$-subalgebra of $B(H)$, where $H$ is a Hilbert space.
For $x \in B(H)$ we define an operator $M_x : \A \to B(H) : a \mapsto x^* a x$.
In this section, we study domination of, and by, multiplication operators,
in relation to compactness. First, record some consequences of the results
from Section \ref{ss:intervals}.

\begin{proposition}\label{prop:mult from A}
Suppose $x$ is an element of a $C^*$-algebra $\A$.
\begin{enumerate}
\item
If $M_x$ is weakly compact, and $0 \leq T \leq M_x : \A \to \A$, then $T$ is compact.
\item
If $0 \leq M_x \leq S : \A \to \A$, and $S$ is weakly compact, then $M_x$ is compact.
\end{enumerate}
\end{proposition}

\begin{proof}
By passing to the second adjoint if necessary, we can assume $\A$ is a
von Neumann algebra. Note that $[0,x^*x] = M_x(\ball(\A)_+)$.
Indeed, if $a \in \ball(\A)_+$, then $0 \leq a \leq \one$, hence
$0 \leq M_x a \leq M_x \one = x^* x$, hence $M_x a \in [0, x^* x]$.
Next show that any $b \in [0, x^* x]$ belongs to $M_x a \in [0, x^* x]$.
By \cite[p. 11]{DixVNA}, there exists $v \in \ball(\A)$ so that
$b^{1/2} = v c$, where $c = (x^* x)^{1/2}$. Write $x = u c$,
where $u$ is a partial isometry from $(\ker x)^\perp$ onto
$\overline{\ran x}$. Then $c = u^* x = x^* u$, and therefore,
$b = M_x(u v^* v u^*)$.

Therefore, $M_x$ is (weakly) compact if and only if the interval
$[0,x^*x]$ is (weakly) compact. By Proposition \ref{prop:comp_int},
the compactness and weak compactness of $[0, x^*x]$
are equivalent. To establish (1), suppose
$0 \leq T \leq M_x$, and $M_x$ is weakly compact. Then
$T(\ball(\A)_+)$ is relatively compact, as a subset of $[0, x^* x]$.
Thus, $T$ is compact. (2) is established similarly.
\end{proof}

If the ``symbol'' $x$ of the operator $M_x$ comes from the ambient $B(H)$,
we obtain:

\begin{proposition}\label{prop:below_multipl}
Suppose $\A$ is an irreducible $C^*$-subalgebra of $B(H)$, $x \in B(H)$,
$M_x : \A \to B(H)$ is compact, and $0 \leq T \leq M_x$.
Then $T$ is compact.
\end{proposition}

\begin{proposition}\label{prop:above_multipl}
Suppose $\A$ is an irreducible $C^*$-subalgebra of $B(H)$,
$S : \A \to B(H)$ is compact, $x \in B(H)$, and $0 \leq M_x \leq S$.
Then $M_x$ is compact.
\end{proposition}

\begin{remark}\label{rem:irreduc}
The irreducibility of $\A$ is essential here. Below we construct an abelian
$C^*$-subalgebra $\A \subset B(H)$, and operators $x_1, x_2 \in B(H)$,
so that $0 \leq M_{x_1} \leq M_{x_2}$, $M_{x_2}$ is compact, while
$M_{x_1}$ is not (here, $M_{x_1}$ and $M_{x_2}$ are viewed as
taking $\A$ to $B(H)$). By \cite{Wi}, there
exist operators $0 \leq R_1 \leq R_2 : C[0,1] \to C[0,1]$ so that
$R_2$ is compact, and $R_1$ is not. Let $\lambda$ be the usual
Lebesgue measure on $[0,1]$, and let $j : C[0,1] \to B(L_2(\lambda))$
be the diagonal embedding (taking a function $f$ to the multiplication
operator $\phi \mapsto \phi f$). By \cite[Theorem 3.11]{Pa}, $R_1$ and $R_2$
are completely positive. Thus, by Stinespring Theorem, these operators
can be represented as $R_i(f) = V_i^* \pi_i(f) V_i$ ($i = 1,2$),
where $\pi_i : C[0,1] \to B(H_i)$ are representations, and
$V_i \in B(L_2(\lambda), H_i)$. Let $H = L_2(\lambda) \oplus_2 H_1 \oplus_2 H_2$.
Then $\pi = j \oplus \pi_1 \oplus \pi_2 : C[0,1] \to B(H)$
is an isometric representation. Let $\A = \pi(C[0,1])$. Furthermore,
consider operators $x_1$ and $x_2$ on $H$, defined via
$$
x_1 = \begin{pmatrix}  0 & 0 & 0 \\ V_1 & 0 & 0 \\ 0 & 0 & 0  \end{pmatrix}
\, \, {\mathrm{and}} \, \,
x_2 = \begin{pmatrix}  0 & 0 & 0 \\ 0 & 0 & 0 \\ V_2 & 0 & 0  \end{pmatrix} .
$$
Then, for any $f \in C[0,1]$, $j R_i(f) = x_i^* \pi(f) x_i$.
Considering $M_{x_1}$ and $M_{x_2}$ as
operators on $\A$, we see that $0 \leq M_{x_1} \leq M_{x_2}$,
$M_{x_2}$ is compact, and $M_{x_1}$ is not.
\end{remark}


The following lemma establishes
a criterion for compactness of $M_x$. This result may be known to experts,
but we could not find any references in the literature.

\begin{lemma}\label{lem:compact}
Suppose $\A$ is an irreducible $C^*$-subalgebra of $B(H)$, and $c \in B(H)$.
Then $c^* \ball(\A)_+ c$ is a relatively compact set if and only if
$c$ is a compact operator.
\end{lemma}

\begin{proof}
By polar decomposition, it suffices to consider the case of $c \geq 0$.
Indeed, write $c = d u$, where $d = (c c^*)^{1/2}$, and $u$ is a
partial isometry from $(\ker c)^\perp = \overline{\ran c^*}$ to
$(\ker c^*)^\perp = \overline{\ran c}$. Then $M_c = M_u M_d$,
and $M_d = M_{u^*} M_c$ (here, we abuse the notation slightly, and
allow $M_u$ and $M_{u^*}$ to act on $B(H)$). Therefore, the sets
$c^* \ball(\A)_+ c = M_c(\ball(\A)_+)$ and
$d \ball(\A)_+ d = M_d(\ball(\A)_+)$ are compact simultaneously.

If $c$ is compact, then, by \cite{Yl72PAMS}, $c \ball(B(H)) c$ is relatively
compact. The set $c \ball(\A)_+ c$ is also relatively compact, since
it is contained in $c \ball(B(H)) c$.

Now suppose $c$ is not compact. By scaling, we
can assume that the spectral projection $p = \chi_{(1,\infty)}(c)$ has
infinite rank. We shall show that, for every $n \in \N$, there exist
$a_1, \ldots, a_n \in \ball(\A)_+$ so that
$\|c(a_i - a_j)c\| > 1/3$ for $i \neq j$. Note first that there exist
mutually orthogonal unit vectors $\xi_1, \ldots, \xi_n$ in $\ran p$,
so that $\langle \xi_i, \xi_j \rangle = \langle c \xi_i, c \xi_j \rangle = 0$
whenever $i \neq j$. Indeed, if $\sigma(c) \cap (1,\infty)$ is infinite,
then there exist disjoint Borel sets $E_i \subset (1,\infty)$ ($1 \leq i \leq n$),
so that $\sigma(c) \cap E_i \neq \infty$. Then we can take $\xi_i \in \chi_{E_i}(c)$.
On the other hand, if $\sigma(c) \cap (1,\infty)$ is finite, then for some
$s \in \sigma(c) \cap (1,\infty)$, the projection $q = \chi_{\{s\}}(c)$
has infinite rank. Then we can take $\xi_1, \ldots, \xi_n \in \ran q$.

Let $\eta_i = c \xi_i/\|c \xi_i\|$ (by construction,
these vectors are mutually orthogonal).
As $\A$ is irreducible, its second commutant is $B(H)$. By
Kaplansky Density Theorem (see e.g.~\cite[Theorem I.7.3]{Da}),
$\ball(\A)_+$ is strongly dense in $\ball(B(H))_+$.
Thus, for every $1 \leq i \leq n$ there exist
$a_i \in \ball(\A)_+$ so that $\|a_i \eta_k\| < 1/3$ for $i \neq k$, and
$\|a_i \eta_i - \eta_i\| < 1/3$. Consider
$b_i = c a_i c \in c (\ball(\A)_+) c$.
For $i \neq j$,
$$
\|b_i - b_j\| \geq \langle c (a_i - a_j) c \xi_i, \xi_i \rangle =
\|c \xi_i\|^2 \langle(a_i - a_j) \eta_i, \eta_i\rangle >
\frac{2}{3} - \frac{1}{3} = \frac{1}{3} .
$$
As $n$ is arbitrary, we conclude that $c (\ball(\A)_+) c$ is not
relatively compact.
\end{proof}

\begin{proof}[Proof of Proposition \ref{prop:below_multipl}]
Suppose $x \in B(H)$ is such that $M_x : \A \to B(H)$ is compact.
By polar decomposition, we can assume that $x \geq 0$.
Then $x \ball(A)_+ x$ is relatively compact, and therefore,
By Lemma \ref{lem:compact}, $x$ is a compact operator.
By Proposition \ref{prop:comp_int}, $[0, x^2]$ is compact.
But $T(\ball(\A)_+) \subset [0,x^2]$, hence $T(\ball(\A)_+)$
is relatively compact. By polarization, $T(\ball(\A))$ is compact.
\end{proof}

To prove Proposition \ref{prop:above_multipl}, we need a technical result.

\begin{lemma}\label{lem:comparison}
Suppose $z \in B(H)$, and $x, y \in [0, \one_H]$. Then
$z x z^* \geq zxyxz^*$.
\end{lemma}

\begin{proof}
Note that $z x z^* - z x y x z^* = z (x - x^2) z^* + z x (\one - y) x z^*$,
and both terms on the right are positive.
\end{proof}

\begin{proof}[Proof of Proposition \ref{prop:above_multipl}]
As in the proof of Proposition \ref{prop:below_multipl}, we can assume
that $x \geq 0$, and that $p = \chi_{(1,\infty)}(x)$ is a
projection of infinite rank. It suffices to show that
there exist $a_0 \geq a_1 \geq \ldots \geq a_n$ in $\ball(\A)_+$,
so that $\|x(a_{k-1} - a_k)x\| > 2/3$ for $1 \leq k \leq n$.
Indeed, if $S$ is compact, then there exist
$u_1, \ldots, u_m \in B(H)$, so that for every $a \in \ball(\A)_+$
there exists $j \in \{1, \ldots, m\}$ so that $\|Sa - u_j\| < 1/3$.
By the pigeon-hole principle, if $n > m$, there exist
$i < j$ in $\{1, \ldots, n\}$ and $k$ in $\{1, \ldots, m\}$, so that
$\max\{\|Sa_i - u_k\|, \|S a_j - u_k\|\} < 1/3$. However,
$\|S a_i - S a_j\| \geq \|x(a_i - a_j)x\| > 2/3$, leading to a contradiction.

Imitating the proof of Proposition \ref{prop:below_multipl}, we use the
spectral decomposition of $x$ to find mutually orthogonal unit vectors
$\xi_1, \ldots, \xi_n$ in $\ran p$, so that (i) $x^k \xi_i$ is orthogonal
to $x^\ell \xi_j$ for any $i \neq j$, and $k, \ell \in \{0,1,\ldots\}$, and
(ii) for any $i$,
$1 = \|\xi_i\| \leq \|x \xi_i\| \leq \|x^2 \xi_i\| \leq \ldots$.
To construct $a_0, \ldots, a_n$, let $c = (2/3)^{1/(2n+1)}$, and
let $\eta_i = x \xi_i/\|x \xi_i\|$. By Kaplansky Density Theorem,
for $0 \leq k \leq n$ there exist $b_k \in \ball(\A)_+$, so that
$$
b_k \eta_i = \left\{ \begin{array}{ll}
   c \eta_i  &  1 \leq i \leq n-k   \\
   0        &  i > n-k
\end{array} \right.
$$
(we can take $b_n = 0$). Let $a_0 = b_0$, $a_1 = b_0 b_1 b_0$,
$a_2 = b_0 b_1 b_2 b_1 b_0$, etc.. By Lemma~\ref{lem:comparison},
$a_0 \geq a_1 \geq \ldots \geq a_n$. Furthermore,
$$
a_k \eta_i = \left\{ \begin{array}{ll}
   c^{2k-1} \eta_i  &  1 \leq i \leq n-k   \\
   0               &  i > n-k
\end{array} \right. ,
$$
and therefore,
$$  \eqalign{
\|x(a_{k-1} - a_k)x\|
&
\geq
 \langle x (a_{k-1} - a_k) x \xi_{n-k+1}, \xi_{n-k+1} \rangle
\cr
&
=
 \langle (a_{k-1} - a_k) \eta_{n-k+1}, \eta_{n-k+1} \rangle =
 c^{2k-1} > \frac{2}{3} .
}  $$
Therefore, the sequence $(a_k)_{k=0}^n$ has the desired properties.
\end{proof}

\subsection{Dunford-Pettis Schur multipliers}\label{ss:DP_mult}

Recall that a map $T : \is_{\mathcal{F}} \to \se$ is called a
\emph{Schur} (or \emph{Hadamard}) \emph{multiplier} if it
can be written in the coordinate form, as $(T x)_{ij} = \phi_{ij} x_{ij}$.
The infinite matrix $\phi$ is called the \emph{symbol} of $T$,
which we denote by $\schu_\phi$. The main goal of this section
is to prove:

\begin{theorem}\label{thm:sch_mult}
Suppose $0 \leq \schu_\phi \leq \schu_\psi$ are Schur multipliers
from $\is_1$ to $\se$ ($\ce$ is a symmetric sequence space).
If $\schu_\psi$ is Dunford-Pettis, then
the same is true for $\schu_\phi$.
\end{theorem}

Recall that an operator is called \emph{Dunford-Pettis} if it maps weakly
null sequences to norm null ones. Equivalently, it carries relatively
weakly compact sets to relatively norm compact sets. The reader is referred
to e.g.~\cite[Section 5.4]{AB} for more information.

The proof relies on several technical lemmas, which may be known to experts.

\begin{lemma}\label{lem:w_null_l1}
A bounded sequence $(x_n)$ in $\is_1$ is weakly null if and only if the following
two conditions are satisfied: (1) $\lim_m \sup_n \|R_m x_n\| = 0$, and
(2) for every $m$, $\lim_n \|Q_m x_n\| = 0$.
\end{lemma}

\begin{proof}
Suppose first $(x_n)$ is weakly null. As $Q_m$ has finite rank, (2) must be
satisfied. If (1) fails, then one can assume, by passing to a subsequence,
that there exists $c > 0$, and a sequence $n_1 < n_2 < \ldots$, so that,
for every $k$, $\|Q_{n_{k+1}} R_{n_k} x_k\| > c$, while
$\|R_{n_{k+1}} x_k\| + \|Q_{n_k} x_k\| < 10^{-k} c$.
Consider the block-diagonal truncation
$P : \is_1 \to \is_1 : x \mapsto \sum_k Q_{n_{k+1}} R_{n_k} x$.
Clearly, $P$ is contractive. Letting, for every $k$,
$y_k = Q_{n_{k+1}} R_{n_k} x_k$, we see that $\|P x_k - y_k\| < 10^{-k} c$.
Thus, for every sequence $(\alpha_k)$,
$$
\|\sum_k \alpha_k x_k\| \geq \|\sum_k \alpha_k y_k\| -
\sum_k |\alpha_k| \cdot 10^{-k} c > \frac{c}{2} \sum_k |\alpha_k| .
$$
Thus, the sequence $(x_k)$ is equivalent to the canonical basis of $\ell_1$,
hence not weakly null.

Now suppose (1) and (2) are satisfied for a bounded sequence $(x_n)$,
and show that, for any $f \in B(\ell_2)$, $\lim_n f(x_n) = 0$.
Indeed, otherwise, by passing to a subsequence, and by scaling,
we can assume that $\sup_n \|x_n\| \leq 1$, and
there exists $f \in \ball(B(\ell_2))$ so that
$\inf_n |f(x_n)| > c$.
Pick $m$ so that $\sup_n \|R_m x_n\| < c/5$.
Note that there exists $M > m$ so that $\|(I - Q_M) (I - R_m) f\| < c/5$.
Indeed,
$$
(I - Q_M) (I - R_m) f = P_M^\perp f P_m + P_m f P_M^\perp .
$$
For a fixed $m$, $B(\ell_2) P_m$ is isomorphic to a Hilbert space.
For every $y \in B(\ell_2) P_m$, $P_M^\perp y \to 0$, hence
$\lim_M P_M^\perp f P_m = 0$. Similarly, $\lim_M P_m f P_M^\perp = 0$.

Finally, pick $N$ so that, for $n > N$, $\|Q_M x_n\| < c/5$. As
$$  \eqalign{
&
\langle f , x_n \rangle =
\Big\langle f ,\big( R_m + (I - R_m) Q_M + (I - Q_M) (I - R_m) \big) x_n \Big\rangle
\cr
&
=
\langle f, R_m x_n \rangle + \langle (I - R_m) f, Q_M x_n \rangle + 
\langle (I - Q_M) (I - R_m) f, x_n \rangle ,
}  $$
we have, for $n > N$,
$$
c < | \langle f, x_n \rangle |
\leq \|R_m x_n\| + 2 \|Q_M x_n\| + \|(I - Q_M) (I - R_m) f\| <
\frac{4c}{5} ,
$$
a contradiction.
\end{proof}

\begin{corollary}\label{cor:DP_from_S1}
An operator $T : \is_1 \to X$ is Dunford-Pettis if and only if,
for every $i$, the restrictions of $T$ to $\span[E_{ij} : j \in \N]$
and $\span[E_{ji} : j \in \N]$ are compact.
\end{corollary}

\begin{proof}
Suppose the restrictions of $T$ to $\span[E_{ij} : j \in \N]$
and $\span[E_{ji} : j \in \N]$ are compact, and $(x_n)$ is a
weakly null sequence in $\is_1$. We have to show that, for every
$c > 0$, $\|T x_n\| < c$ for $n$ large enough. Without loss of generality,
assume $T$ is a contraction, and $\sup_n \|x_n\| \leq 1$. Find $M > m$
so that $\sup_n \|R_m x_n\| < c/4$, and
$$
\|T|_{\span[E_{ij} : j > M]}\| + \|T|_{\span[E_{ji} : j > M]}\| < \frac{c}{4M} .
$$
Find $N \in \N$ so that $\sup_{n > N} \|Q_M x_n\| < c/4$.
Thus, for $n > N$, $\|T x_n\| < 3c/4$.

Conversely, suppose $T$ is Dunford-Pettis, but its restriction to
$\span[E_{ij} : j \in \N]$ is not compact. Then there exist
$n_1 < n_2 < \ldots$, and $\alpha_j \in \C$, so that the vectors
$x_k = \sum_{j = n_k+1}^{n_{k+1}} \alpha_j E_{ij}$, so that
$\|x_k\| = 1$, and $\limsup_k \|T x_k\| > 0$. However, the sequence
$(x_k)$ is weakly null, while the sequence $(T x_k)$ is not norm null,
yielding a contradiction. The restrictions to $\span[E_{ji} : j \in \N]$
are handled similarly.
\end{proof}

Specializing the previous result to Schur multipliers,
we immediately obtain:

\begin{corollary}\label{cor:DP_Schur}
A Schur multiplier with the symbol $\phi$, acting from $\is_1$
to $\se$, is Dunford-Pettis if and only if, for any $i$,
$\lim_j \phi_{ij} = \lim_j \phi_{ji} = 0$.
\end{corollary}

\begin{proof}
By Corollary \ref{cor:DP_from_S1}, $\schu_\phi : \is_1 \to \se$ is
Dunford-Pettis iff, for every $i$, the restrictions of $\schu_\phi$ to
$\span[E_{ij} : j \in \N]$ and $\span[E_{ji} : j \in \N]$ are compact.
By the definition, $\schu_\phi$ maps $E_{ij}$ to $\phi_{ij} E_{ij}$.
It is well known that, for any $\ce$, $\span[E_{ij} : j \in \N] \subset \se$
is isometric to $\ell_2$, via an isometry sending the matrix units $E_{ij}$ to the
elements of the orthonormal basis. Thus, $\schu_\phi|_{\span[E_{ij} : j \in \N]}$ is
compact iff $\lim_j \phi_{ij} = 0$. Similarly,
$\schu_\phi|_{\span[E_{ji} : j \in \N]}$ is compact iff $\lim_j \phi_{ji} = 0$.
\end{proof}

\begin{lemma}\label{lem:matrices}
Suppose $c > 0$ and $m \in \N$ satisfy $(mc)^2 > m+1$. Suppose,
furthermore, that $C$ and $D$ are positive matrices, with entries
$C_{ij}$ and $D_{ij}$ ($0 \leq i,j \leq m$), respectively,
so that $\max_{i,j} \{\max\{ |C_{ij}|, |D_{ij}|\}\} \leq 1$,
$|C_{0j}| > c$ for $1 \leq j \leq m$, and
$|D_{ij}| < 10^{-2(i+j)}$ for $i \neq j$. Then
the inequality $C \leq D$ cannot hold.
\end{lemma}

\begin{proof}
Suppose, for the sake of contradiction, that $D \geq C$.
Then, for any vector $\xi \in \ell_2^{m+1}$,
$$
\|D^{1/2} \xi\|^2 = \langle D^{1/2} \xi, D^{1/2} \xi \rangle =
\langle D \xi, \xi \rangle \geq \langle C \xi, \xi \rangle =
\|C^{1/2} \xi\|^2 ,
$$
hence there exists a contraction $U$ so that $U D^{1/2} \xi = C^{1/2} \xi$.
Thus, $C = D^{1/2} U^* U D^{1/2}$. By \cite[Lemma 1.21]{Zh}, the
block matrix $\begin{pmatrix} D & C \\ C & D \end{pmatrix}$ is positive.

Denote the canonical basis in $\ell_2^{m+1}$ by $(e_i)_{i=0}^m$. Consider
the vector $\xi = \begin{pmatrix} \xi_1 \\ \xi_2 \end{pmatrix} \in \ell_2^{2(m+1)}$,
where $\xi_1 = \alpha e_0$, and $\xi_2 = - \sum_{i=1}^m \omega_i e_i$.
Here, $\omega_i = C_{i0}/|C_{i0}|$, and
$\alpha = mc$. By the above,
\begin{equation}
0 \leq
\Big\langle \begin{pmatrix} D & C \\ C & D \end{pmatrix} \xi, \xi \Big\rangle =
\langle D \xi_1, \xi_1 \rangle + \langle D \xi_2, \xi_2 \rangle +
2 \re \langle C \xi_1, \xi_2 \rangle .
\label{eq:positive}
\end{equation}
Note that $\langle D \xi_1, \xi_1 \rangle = \alpha^2 D_{00} \leq \alpha^2$, and
$$
\langle D \xi_2, \xi_2 \rangle \leq
\sum_{i=1}^m D_{ii} + 2 \sum_{1 \leq i < j \leq m} |D_{ij}| \leq
m + 2 \sum_{1 \leq i < j \leq m} 10^{-2(i+j)} < m + 1 .
$$
On the other hand,
$$
\langle C \xi_1, \xi_2 \rangle =
- \alpha \sum_{i=1}^m C_{i0} \cdot \overline{\frac{C_{i0}}{|C_{i0}|}} <
- \alpha m c .
$$
Returning to \eqref{eq:positive}, we see that
$$
\Big\langle \begin{pmatrix} D & C \\ C & D \end{pmatrix} \xi, \xi \Big\rangle \leq
\alpha^2 + m + 1 - 2 \alpha mc < 0,
$$
a contradiction.
\end{proof}

\begin{proof}[Proof of Theorem \ref{thm:sch_mult}]
We say that an infinite matrix $\phi$ is \emph{formally positive}
if each of its finite submatrices is positive. By \cite[Theorem 3.7]{Pa},
$\schu_\sigma \geq 0$ iff $\sigma$ is formally positive.

Suppose, for the sake of contradiction, that
$0 \leq \schu_\phi \leq \schu_\psi$, where $\schu_\psi$ is
Dunford-Pettis, while $\schu_\phi$ is not. We can assume that
$\schu_\psi$ is contractive, hence, for any $(i,j)$,
$\max\{|\phi_{ij}|, |\psi_{ij}|\} \leq 1$.
Corollary \ref{cor:DP_Schur} shows that, for any
$i$, $\lim_{j \to \infty} \psi_{ij} = 0$. By rearranging rows
and columns if necessary, we can assume the existence of
$n_0 < n_1 < n_2 < \ldots$, so that $|\phi_{n_0 n_k}| > c > 0$.
Passing to further subsequence, we obtain
$|\psi_{n_i n_j}| < 10^{-2(i+j)}$ for $i \neq j$.

Now select $m$ so that $mc > 4(m+1)$, and define matrices $C$ and $D$,
with entries $C_{ij} = \phi_{n_i n_j}$ and $D_{ij} = \psi_{n_i n_j}$
($0 \leq i,j \leq m$), respectively. As noted above, the matrices
$C$ and $D$ are positive. By Lemma~\ref{lem:matrices}, we cannot
have $C \leq D$. Thus, a contradiction.
\end{proof}

\subsection{Weakly compact operators}\label{ss:WC}

In this section, we show that, under certain conditions, weak compactness
is inherited under domination.
First consider operators on $C^*$-algebras and their duals.


\begin{theorem}\label{th:weak_comp}
Suppose $E$ is an OBS, and $\A$ is a $C^*$-algebra,
$S$ is a weakly compact operator, and one of the following holds:
\begin{enumerate}
\item
$E$ is generating, and $0 \leq T \leq S : E \to \A^\star$.
\item
$E$ is normal, and $0 \leq T \leq S : \A \to E$.
\end{enumerate}
Then $T$ is weakly compact.
\end{theorem}

Note that, for commutative $\A$, this theorem follows from
\cite{Wi81}, and the order continuity of $\A^\star$.

\begin{proof}
(1)
Suppose, for the sake of contradiction, that $T(\ball(E)_+)$
is not weakly compact.
By Pfitzner\rq{s} Theorem \cite{Pf},
there exist a bounded sequence $(a_n) \subset \A$
of positive pairwise orthogonal elements, a sequence
$(\phi_n) \subset \ball(E)_+$, and $c>0$, such that $T \phi_n (a_n) > c$.
Therefore, $S \phi_n(a_n) > c$,  which contradicts the weak compactness
of $S(\ball(E))$ (once again, by Pfitzner\rq{s} Theorem).

(2)
Apply part (1) to $0 \leq T^\star \leq S^\star$.
\end{proof}

\begin{remark}\label{rem:weak_comp}
Theorem \ref{th:weak_comp} fails for operators from duals
of $C^*$-algebras to $C^*$-algebras,
even in the commutative setting.
Indeed, by \cite[Theorem 5.31]{AB}, there exist
$0 \leq T \leq S : \ell_1 \to \ell_\infty$, so that $S$
is weakly compact, whereas $T$ is not.
\end{remark}

For operators to or from general Banach lattices, we have:

\begin{theorem}\label{th:weak_dom}
Suppose either $(i)$ $A$ is a generating OBS,
and $B$ is order continuous Banach lattice,
or $(ii)$ $A$ is a Banach lattice with order continuous dual, and $B$ is an
normal OBS. If $0 \le T \le S:A \to B$, and $S$ is weakly compact,
then $T$ is weakly compact as well.
\end{theorem}

\begin{proof}
The proof of $(i)$ is contained in the first few lines of
the proof of \cite[Theorem~5.31]{AB}. $(ii)$ follows by
duality.
\end{proof}


Next we obtain a partial generalization of the above results for
non-commutative function spaces. In the discrete case, we obtain a
characterization of order continuous Banach lattices.


\begin{proposition}\label{prop:se_criterion}
Suppose $\ce$ is a symmetric sequence space, containing
a copy of $\ell_1$, $H$ is an infinite dimensional Hilbert space,
and $X$ is a Banach lattice.
Then the following are equivalent:
\begin{enumerate}
\item If $0 \le T \le S:\se(H) \to X$, and $S$ is weakly compact,
 then $T$ is weakly  compact.
\item $X$ is order continuous.
\end{enumerate}
\end{proposition}

\begin{proof} 
$(2) \Rightarrow (1)$ follows from Theorem~\ref{th:weak_dom}.

$(1) \Rightarrow (2)$: By Proposition~\ref{prop:l_1_S_E_cont}
$\se(H)$ contains a positive disjoint sequence, that spans a positively
complemented copy of $\ell_1$. Hence, the result follows from
\cite[Theorem~5.31]{AB}.
\end{proof}

Now consider domination by a weakly compact operator for
non-commutative function spaces.

Recall that a non-commutative symmetric function space $\ce$
is said to have the \emph{Fatou Property} (sometimes referred to
as the \emph{Beppo Levi Property}) if for any norm-bounded increasing net
$(x_i) \subset \ce_+$, there exists $x \in \ce$ so that $x_i \uparrow x$, and
$\|x\| = \sup_i \|x_i\|$. In the commutative setting, any symmetric
space with the Fatou Property is order complete.

We say that a non-commutative function space $\ce$ is a \emph{KB space}
if any increasing norm bounded sequence in $\ce$ is
norm-convergent. Equivalently, $\ce$ is order continuous, and
has the Fatou Property (see \cite{DdP11}). Furthermore, the following
are equivalent: (i) $\ce$ is a KB space, (ii) $\ce$ is weakly
sequentially complete, and (iii) $\ce$ contains no copy of $c_0$.
It is clear from \cite{DDdP93} that, if $\ce$ is symmetric KB function
space, then the same is true of $\ce(\tau)$.



The following result is contained in \cite[Section 5]{DDdP93}.

\begin{proposition}\label{prop:DDdP}
Suppose $\ce$ is a non-commutative strongly symmetric
function space. Then:
\begin{enumerate}
\item
$\ce^\times$ is strongly symmetric,
\item
$\ce^\times$ coincides with $\ce^\star$ if and only if $\ce$ is order continuous.
In this case, for every $f \in \ce^\star$ there exists a unique $y \in \ce^\times$
so that $f(x) = \tau(xy)$, for any $x \in \ce$.
\item
$\ce$ coincides with $\ce^{\times\times}$ if and only if $\ce$ has the Fatou Property.
\end{enumerate}
\end{proposition}


\begin{proposition}\label{prop:to_order_cont}
Suppose $\ce = \ce(\tau)$ is a non-commutative strongly symmetric KB
function space, $X$ a generating OBS, and
$0 \leq T \leq S : X \to \ce$, with $S$ weakly compact.
Then $T$ is weakly compact as well.
\end{proposition}

\begin{proof}
By \cite[Section 5]{DDdP93}, any positive element
$\phi \in \ce^{\star\star} = (\ce^\times)^\star$ can be written as
$\phi(f) = \tau(a f) + \psi(f)$, where $a \in \ce$ is positive,
and $\psi$ is a positive singular functional. The canonical
embedding of $\ce$ into its double dual takes $a$ to the
linear functional $f \mapsto \tau(fa)$.

$S$ is weakly compact, hence $S^{\star\star}(X) \subset \ce$.
A normal functional cannot dominate a singular one,
hence $T^{\star\star}(\ball(X^{\star\star})_+) \subset \ce$.
As noted in Section~\ref{ss:intro}, $X^{\star\star}$ is a generating OBS,
hence $T^{\star\star}(\ball(X^{\star\star})) \subset \ce$.
Therefore, $T$ is weakly compact.
\end{proof}

Alternatively, one can prove the above result using the characterization of
$\sigma({\mathcal{F}}^\times,{\mathcal{F}})$-compact sets given in
\cite[Proposition 6.2]{DdP09}.

\begin{remark}\label{rem:ord_cont}
Note that the assumptions of Proposition \ref{prop:to_order_cont}
are stronger than those of its commutative counterpart --
Theorem \ref{th:weak_dom}. For instance, the statement of
Theorem \ref{th:weak_dom}(i) holds when the range space is
order continuous. Propositions \ref{prop:to_order_cont} is
proved under the additional assumption of the Fatou property.
One reason for this is that much more is known about order
continuous Banach lattices (see e.g. \cite[Section 2.4]{M-N}).
One useful characterization states that a Banach lattice $\ce$
is order continuous iff it is an ideal in its second dual.
No such description seems to be known in the non-commutative setting.
\end{remark}

\section{Miscellaneous results}\label{sec:misc}

\subsection{$2$-positivity and decomposability: negative results}\label{ss:positive}

In this section we consider stronger versions of positivity,
such as $2$-positivity and indecomposability, as well as
the appropriate notions of domination. We show that these properties
are not, in general, inherited by the dominated operator.

\begin{proposition}\label{prop:nogo}
(a) There are $0 \leq T \leq_c S$, acting on $M_2$, so that $S$ is completely positive,
but $T$ is not $2$-positive.

(b) There are $0 \leq T \leq_c S$, acting on $M_3$, so that $S$ is completely positive,
but $T$ is not decomposable.
\end{proposition}

For the definition and basic properties of decomposable maps, see e.g.~\cite{St82}.
Note that part (b) is optimal in the sense that any positive map from $M_2$ to $M_3$
is decomposable \cite{Wo}.

In the proof below, we use the notation $E_{ij}$ for the matrix with $1$ in the $(i,j)$
position, and $0$'s elsewhere.

\begin{proof}
(a) Define $T(a) = a^t$, and $S(a) = \trace(a) \one$
($\trace( \cdot )$ stands for the canonical trace on $M_2$).
Clearly, $T \geq 0$, and $S$ is completely
positive. Indeed, consider
$a = \sum_{i,j=1}^n E_{ij} \otimes a^{(ij)} \in M_n(M_2) \geq 0$
(here, $a^{(ij)} = (a^{ij}_{k\ell})_{k,\ell=1}^2 \in M_2$). 
Passing to submatrices, we see that
for $k=1,2$, the $n \times n$ matrix $a^\prime_k = (a^{(ij)}_{kk})$ is positive.
Thus, $(I_{M_n} \otimes S) a =
 (a^\prime_1 + a^\prime_2) \otimes (E_{11} + E_{22}) \geq 0$.

The fact that $T$ is not $2$-positive is folklore: just apply $I_{M_2} \otimes T$ to
$\sum_{i,j=1}^2 E_{ij} \otimes E_{ij}$.
To establish that $S-T \geq_{cp} 0$, note that $(S-T)(a) = u a u^*$,
where $u = \begin{pmatrix} 0 & 1 \\ -1 & 0 \end{pmatrix}$.

(b) Define 
$$
U \begin{pmatrix}
   a_{11}  &  a_{12}  &  a_{13}  \\
   a_{21}  &  a_{22}  &  a_{23}  \\
   a_{31}  &  a_{32}  &  a_{33}
\end{pmatrix} =
\begin{pmatrix}
   a_{11}    &  - a_{12}  &  - a_{13}  \\
   - a_{21}  &  a_{22}    &  - a_{23}  \\
   - a_{31}  &  - a_{32}  &  a_{33}
\end{pmatrix} ,
$$
$$
V \begin{pmatrix}
   a_{11}  &  a_{12}  &  a_{13}  \\
   a_{21}  &  a_{22}  &  a_{23}  \\
   a_{31}  &  a_{32}  &  a_{33}
\end{pmatrix} =
\begin{pmatrix}
   a_{33}  &  0         &  0            \\
   0         &  a_{11}  &  0            \\
   0         &  0          &  a_{22}
\end{pmatrix} ,
$$
$$
W \begin{pmatrix}
   a_{11}  &  a_{12}  &  a_{13}  \\
   a_{21}  &  a_{22}  &  a_{23}  \\
   a_{31}  &  a_{32}  &  a_{33}
\end{pmatrix} =
\begin{pmatrix}
   a_{11}   &  0         &  0            \\
   0        &  a_{22}    &  0            \\
   0        &  0         &  a_{33}
\end{pmatrix} .
$$
Let $T = U + V$, and $S = V + 2W$.
By \cite{St82}, $T$ is positive, but not decomposable. On the other hand,
the maps $V$ and $W$ are completely positive, hence so is $S$. Furthermore,
$S - T = I$ (the identity map on $M_3$), hence it is completely positive as well.
\end{proof}

For powers of operators, we get:

\begin{proposition}\label{prop:nogo powers}
There are $0 \leq T \leq_c S$, acting on $M_2$, so that $S$ is completely positive,
while $T$ is not $2$-positive, and $T = T^2$.
\end{proposition}

\begin{proof}
Define $T(a) = (a+a^t)/2$, and $S(a) = (\trace(a) \one + a)/2$.
As in the proof of Proposition~\ref{prop:nogo}, we can establish
the inequalities $0 \leq_c S$ and $0 \leq T \leq_c S$.
Clearly, $T  = T^2$. To show that $T$ is not $2$-positive, consider
$x = \sum_{i,j=1}^2 E_{ij} \otimes E_{ij} \in M_2 \otimes M_2$. $x$ can be
identified with the $4 \times 4$ matrix
$$
\begin{pmatrix}
   1 & 0 & 0 & 1  \\
   0 & 0 & 0 & 0  \\
   0 & 0 & 0 & 0  \\
   1 & 0 & 0 & 1
\end{pmatrix} .
$$
Then
$$
(I_{M_2} \otimes T)(x) = \frac{1}{2} \begin{pmatrix}
   2 & 0 & 0 & 1  \\
   0 & 0 & 1 & 0  \\
   0 & 1 & 0 & 0  \\
   1 & 0 & 0 & 2
\end{pmatrix} ,
$$
which is not positive.
\end{proof}

\begin{remark}\label{rem:powers}
It is not clear whether we can strengthen Proposition~\ref{prop:nogo}(b) to make
the powers of $T$ (not just $T$ itself) non-decomposable. The operator $T$ presented
in the proof of Proposition~\ref{prop:nogo}(b) will not work, since $T^2$ is completely
positive. Indeed, \cite{St82} shows that $T = U + \mu V$ is not decomposable for
$\mu \geq 1$. However, $U^2 = I$, and $UV = VU = V$.
Thus, $T^2 = I + 2 \mu V + \mu^2 V^2$, which is completely positive.
\end{remark}

\subsection{A remark on operator systems}\label{ss:op_sys}
In previous section, we were working with non-commutative function spaces,
or with $C^*$-algebras. This brief section shows that general operator
systems have too few positive elements for any results about domination
and inheritance of properties.

Recall that an \emph{operator system} is a subspace of $B(H)$,
closed under conjugation. It is \emph{unital} if it contains $\one$.
if $A$ and $B$ are operator systems, and $T : A \to B$, we say that
$T$ is positive ($T \geq 0$) if $Ta \geq 0$ for any $a \geq 0$. Moreover,
$T$ is completely positive ($T \geq_c 0$) if $T \otimes I_{M_n} \geq 0$
for every $n$. Write $T \geq S$ ($T \geq_c S$) if $T - S \geq 0$
(resp.~$T - S \geq_c 0$).

It turns out that little can be said about domination in operator systems.
More precisely, there exists a unital operator system $A$, and a rank $1$
$S \in B(A)$, so that $I_A \leq_c S$. $A$ may be chosen to be infinite dimensional,
and even non-separable. We describe the construction of $A$ and $S$ below.

Suppose $X \subset B(H)$ is an operator
system (not necessarily unital). Using ``Paulsen's trick'', define $A$ as the
set of all block matrices on $H \oplus_2 H$, of the form
$\begin{pmatrix} \lambda \one_H & x \\ y & \lambda \one_H \end{pmatrix}$,
where $\lambda \in \C$, and $x, y \in X$. It is easy to see that
$\begin{pmatrix} \lambda \one_K & x \\ y & \lambda \one_H \end{pmatrix} \geq 0$
iff $x = y^*$, and $\lambda \geq \|x\|$. Set
$S \begin{pmatrix} \lambda \one_H & x \\ y & \lambda \one_H \end{pmatrix} =
2 \begin{pmatrix} \lambda \one_H & 0 \\ 0 & \lambda \one_H \end{pmatrix} =
2 \lambda \one_{H \oplus_2 H}$.

\begin{proposition}\label{prop:dom_op_sys}
In the above notation, $S \geq_c I_A$.
\end{proposition}

\begin{proof}
It suffices to observe that
$$
(S - I_A) \begin{pmatrix} \lambda \one_H & x \\ y & \lambda \one_H \end{pmatrix} =
\begin{pmatrix} \lambda \one_H & -x \\ -y & \lambda \one_H \end{pmatrix} =
u \begin{pmatrix} \lambda \one_H & x \\ y & \lambda \one_H \end{pmatrix} u ,
$$
with $u = \begin{pmatrix} \one_H & 0 \\ 0 & -\one_H \end{pmatrix}$.
\end{proof}



\begin{thebibliography}{20}


\bibitem{Al} J. Alexander,
`Compact Banach algebras',
{\em Proc. London Math. Soc.} 18 (1968), 1--18.

\bibitem{AB:80} C.D. Aliprantis \and O. Burkinshaw,
`Positive compact operators on Banach Lattices',
{\em Math Z.} 174 (1980), 289--298.

\bibitem{AB:81} C.D. Aliprantis \and O. Burkinshaw,
`On weakly compact operators on Banach lattices',
{\em Proc. Amer. Math. Soc.} 274 (1981), 573--578.

\bibitem{AB} C.D. Aliprantis \and O. Burkinshaw,
{\em Positive operators},
Springer, Dordrecht (2006).

\bibitem{Ble} M. Almus, D. Blecher, and C. Read,
`Ideals and hereditary subalgebras in operator algebras',
preprint.

\bibitem{An} T. Ando,
`On fundamental properties of a Banach space with a cone',
{\em Pacific J. Math.} 12 (1962), 1163--1169.

\bibitem{Ar} J. Arazy,
`Basic Sequences, embeddings, and the uniqueness of
   the symmetric structure in unitary matrix spaces',
{\em J. Funct. Anal.} 40 (1980), 302--340.



\bibitem{BR} C. Batty \and D. Robinson,
`Positive One-Parameter Semigroups on Ordered Banach Spaces',
{\em Acta Applicandae Mathematicae} 1 (1984), 221--296.

\bibitem{Bl} B. Blackadar,
{\em Operator algebras},
Springer-Verlag, Berlin (2006).


\bibitem{BT} M. Bresar \and Yu. Turovskii,
`Compactness conditions for elementary operators',
{\em Studia Math.} 178 (2007), 1--18.


\bibitem{CW} Z. Chen \and A. Wickstead,
`Positive weak compactness of solid hulls in Banach lattices',
{\em Indag. Math.} 9 (1998), 187--196.

\bibitem{CS} V. Chilin \and F. Sukochev,
`Weak convergence in non-commutative symmetric spaces',
{\em J. Operator Theory} 31 (1994), 35--65. 



\bibitem{Da} K. Davidson,
{\em $C^*$-algebras by example},
Amer. Math. Soc., Providence, RI (1996).

\bibitem{Dix} J. Dixmier,
{\em $C^*$-algebras},
North Holland, Amsterdam (1977).

\bibitem{DixVNA} J. Dixmier,
{\em von Neumann algebras},
North Holland, Amsterdam (1981).




\bibitem{DDFr} P. G. Dodds \and D.H. Fremlin,
`Compact operators in Banach lattices',
{\em Israel J. Math.} 34 (1979), 287--320.


\bibitem{DDdP92} P.~Dodds, T.~Dodds, \and B.~de~Pagter,
`Fully symmetric operator spaces',
{\em Int. Eq. Op. Theory} 15 (1992) 942--972.

\bibitem{DDdP93} P.~Dodds, T.~Dodds, \and B.~de~Pagter,
`Non-commutative K\"othe duality',
{\em Trans. Amer. Math. Soc.} 339 (1993) 717--750.



\bibitem{DdP09} P.~Dodds \and B.~de~Pagter,
`Non-commutative Yosida-Hewitt theorems and singular functionals
in symmetric spaces of $\tau$-measurable operators',
in {\em Vector measures, integration and related topics}, 
Oper. Theory Adv. Appl. 201 (2010), 183--198.

\bibitem{DdP10} P.~Dodds \and B.~de~Pagter,
`Completely positive compact operators on non-commutative
symmetric spaces', {\em Positivity} 14 (2010) 665--679.

\bibitem{DdP11} P.~Dodds \and B.~de~Pagter,
`Properties $(u)$ and $(V^*)$ of Pelczynski in symmetric
spaces of $\tau$-measurable operators',
{\em Positivity} 14 (2011) 571--594.

\bibitem{DdP12} P.~Dodds \and B.~de~Pagter,
`The non-commutative Yosida-Hewitt Theorem revisited',
{\em Trans. Amer. Math. Soc.}, to appear.






\bibitem{FK} T. Fack \and H. Kosaki,
`Generalized $s$-numbers of $\tau$-measurable operators',
{\em Pacific J. Math.} 123 (1986), 269--300.

\bibitem{Fa} D. Farenick,
`Irreducible positive linear maps on operator algebras',
{\em Proc. Amer. Math. Soc.} 124 (1996), 3381--3390.


\bibitem{FHT} J. Flores, F. L. Hernandez, \and P. Tradacete,
`Powers of operators dominated by strictly singular operators',
{\em Q. J. Math.} 59 (2008), 321--334.

\bibitem{Fr} Y.~Friedman,
`Subspace of $LC(H)$ and $C_p$',
{\em Proc. Amer. Math. Soc.} 53 (1975), 117--122.

\bibitem{GK} I.~C. Gohberg \and M.~G. Krein,
{\em Introduction to the theory of linear nonselfadjoint operators},
Amer. Math. Soc., Providence, RI (1969).



\bibitem{Hu} T. Huruya,
`A Spectral Characterization of a class of $C^*$-algebras',
{\em Sci. Rep. Niigata Univ. Ser. A} 15 (1978), 21--24.


\bibitem{Je} H. Jensen,
`Scattered $C^*$-algebras',
{\em Math. Scand.} 41 (1977), 308--314.

\bibitem{KaSa} N. Kalton \and P. Saab,
`Ideal properties of regular operators between Banach lattices',
{\em Illinois J. Math.} 29 (1985), 382--400.

\bibitem{KS} N. Kalton \and F. Sukochev,
`Symmetric norms and spaces of operators',
{\em J. Reine Angew. Math.} 621 (2008), 81--121.


\bibitem{KR2} R. Kadison \and J. Ringrose,
{\em Fundamentals of the theory of operator algebras II},
Amer. Math. Soc., Providence, RI (1997).


\bibitem{KM00} A. Kaminska \and M. Mastylo,
`The Dunford-Pettis property for symmetric spaces',
{\em Canad. J. Math.} 52 (2000), 789--803.

\bibitem{KM02} A. Kaminska \and M. Mastylo,
`The Schur and (weak) Dunford-Pettis properties in Banach lattices',
{\em  J. Aust. Math. Soc.} 73 (2002), 251--278.


\bibitem{KPS} S. Krein, Yu. Petunin, \and E. Semenov,
{\em Interpolation of linear operators},
Translated from the Russian by J. Sz\"ucs,
Translations of Mathematical Monographs, 54,
Amer. Math. Soc., Providence, RI (1984).


\bibitem{LT1} J. Lindenstrauss \and L. Tzafriri,
{\em Classical Banach spaces I},
Springer-Verlag, Berlin (1977).

\bibitem{LT2} J. Lindenstrauss \and L. Tzafriri,
{\em Classical Banach spaces II},
Springer-Verlag, Berlin (1979).

\bibitem{MNha} J. Marcolino Nhany,
`La stabilit\'e des espaces $L_p$ non-commutatifs',
{\em Math. Scand.} 81 (1997), 212--218.

\bibitem{M-N} P. Meyer-Nieberg,
{\em Banach lattices},
Springer-Verlag, Berlin (1991).

\bibitem{Mu} T.~Murphy,
`Continuity of positive linear functionals on Banach $\ast$-algebras',
{\em Bull. London Math. Soc.} 1 (1969), 171--173.



\bibitem{deP} B. de Pagter,
`Non-commutative Banach function spaces',
{\em Positivity}, 
Trends Math., Birkh\"auser, Basel (2007),  197--227.

\bibitem{Pa} V. Paulsen,
{\em Completely bounded maps and operator algebras},
Cambridge University Press (2002).


\bibitem{Pf} H. Pfitzner,
`Weak compactness in the dual of a $C^*$-algebra is determined commutatively',
{\em Math. Ann.} 298 (1994), 349--371.



\bibitem{PiLP} G. Pisier,
`$L_p$ spaces', {\em Asterisque} (1998).








\bibitem{Sch} H. H. Schaefer,
{\em Banach lattices and positive operators},
Springer, Berlin (1974).

\bibitem{Si} B. Simon,
{\em Trace ideals and their applications. Second edition},
Amer. Math. Soc., Providence, RI (2005).

\bibitem{Sp} E. Spinu,
`Dominated inessential operators',
{\em J. Math. Anal. Appl.} 383 (2011), 259--264.

\bibitem{St82} E. Stormer,
`Decomposable positive maps on $C^*$-algebras',
{\em Proc. Amer. Math.} 86 (1982), 402--404.


\bibitem{Tak} M. Takesaki,
{\em Theory of operator algebras I},
Springer-Verlag, New York (2003).


\bibitem{Wi81} A. Wickstead,
`Extremal structure of cones of operators',
{\em Quart. J. Math. Oxford Ser. (2)} 126 (1981), 239--253.

\bibitem{Wi} A. Wickstead,
`Converses for the Dodds-Fremlin and Kalton-Saab theorems',
{\em Math. Proc. Cambridge Philos. Soc.} 120 (1996), 175--179.

\bibitem{Wn89a} W. Wnuk,
`A note on the positive Schur property',
{\em Glasgow Math. J.} 31 (1989), 169--172.


\bibitem{Woj} P. Wojtaszczyk,
`On linear properties of separable conjugate spaces of $C^*$-algebras',
{\em Studia Math.} 52 (1974), 143--147.

\bibitem{Wo} S. Woronowicz,
`Positive maps of low dimensional matrix algebras',
{\em Rep. Math. Phys.} 10 (1976), 165--183.


\bibitem{Yl68Suomi} K. Ylinen,
`Compact and finite dimensional elements of normed algebras',
{\em Ann. Acad. Sci. Fenn. Ser. A I Math.}, 408 (1968).

\bibitem{Yl72PAMS} K. Ylinen,
`A note on the compact elements of $C^∗$-algebras',
{\em Proc. Amer. Math. Soc.} 35 (1972), 305--306.

\bibitem{Yl75PAMS} K. Ylinen,
`Weakly compact continuous elements of $C^*$-algebras',
{\em Proc. Amer. Math. Soc.} 52 (1975), 323--326.

\bibitem{Yl75Suomi} K. Ylinen,
`Dual $C^*$-algebras, weakly semi-completely continuous elements,
  and the extreme rays of the positive cone',
{\em Ann. Acad. Sci. Fenn. Ser. A I Math.} 599 (1975).

\bibitem{Zh} X. Zhan,
{\em Matrix inequalities},
Springer-Verlag, New York (2002).

\end{thebibliography}

\end{document}